\documentclass[11pt]{amsart}

\usepackage{amsmath,color,amssymb}
\usepackage{amsthm}
\usepackage{hyperref}
\usepackage[margin=1.0in]{geometry}
\usepackage{dsfont}
\usepackage{enumitem}

\numberwithin{equation}{section}

\newtheorem{prop}{Proposition}
\newtheorem{lemma}[prop]{Lemma}

\newtheorem{thm}[prop]{Theorem}
\newtheorem{cor}[prop]{Corollary}

\numberwithin{prop}{section}

\theoremstyle{definition}
\newtheorem{defn}[prop]{Definition}

\newtheorem{rmk}[prop]{Remark}
\newtheorem*{thm12}{Theorem 1.2}

\renewcommand{\geq}{\geqslant}
\renewcommand{\leq}{\leqslant}

\newcommand{\del}{\partial}
\newcommand{\dt}{\frac{\partial}{\partial t}}
\newcommand{\brs}[1]{\left| #1 \right|}
\newcommand{\G}{\Gamma}
\newcommand{\gG}{\Gamma}
\renewcommand{\gg}{\gamma}
\newcommand{\gD}{\Delta}
\newcommand{\gd}{\delta}

\newcommand{\gl}{\lambda}

\newcommand{\ga}{\alpha}

\newcommand{\N}{\nabla}
\newcommand{\FF}{\mathcal F}

\newcommand{\LL}{\mathcal L}
\newcommand{\MM}{\mathcal M}
\newcommand{\cP}{\mathcal P}
\newcommand{\cyl}{\mathfrak{C}_{T'}}
\newcommand{\til}[1]{\widetilde{#1}}

\newcommand{\R}{\mathbb{R}}

\newcommand{\E}{\mathbb E}

\newcommand{\eps}{\varepsilon}

\newcommand{\IP}[1]{\left<#1\right>}

\newcommand{\HH}{\mathcal{H}}

\newcommand{\hook}{\mathbin{\hbox{\vrule height2.4pt width4.5pt depth-2pt
\vrule height5pt width0.4pt depth-2pt}}}

\DeclareMathOperator{\Rc}{Rc}

\DeclareMathOperator{\tr}{tr}

\DeclareMathOperator{\divg}{div}
\DeclareMathOperator{\grad}{grad}

\begin{document}

\title[Bochner formulas, functional inequalities and generalized Ricci flow]{Bochner formulas, functional inequalities and generalized Ricci flow}

\begin{abstract} As a consequence of the Bochner formula for the Bismut connection acting on gradients, we show sharp universal Poincar\'e and log-Sobolev inequalities along solutions to generalized Ricci flow.  Using the two-form potential we define a twisted connection on spacetime which determines an adapted Brownian motion on the frame bundle, yielding an adapted Malliavin gradient on path space.  We show a Bochner formula for this operator, leading to characterizations of generalized Ricci flow in terms of universal Poincar\'e and log-Sobolev type inequalities for the associated Malliavin gradient and Ornstein-Uhlenbeck operator.
\end{abstract}

\author{Eva Kopfer}
\address{Institut f\"ur Angewandte Mathematik\\
Universit\"at Bonn, 53115 Bonn, Germany}
\email{\href{mailto:eva.kopfer@iam.uni-bonn.de}{eva.kopfer@iam.uni-bonn.de}}

\author{Jeffrey Streets}
\address{Rowland Hall\\
         University of California, Irvine, CA}
\email{\href{mailto:jstreets@uci.edu}{jstreets@uci.edu}}

\date{\today}

\maketitle

\section{Introduction}

In the analysis of Ricci flow, the classic Bochner formula for gradients plays a key role.  This basic formula underlies gradient estimates for solutions to the heat equation along Ricci flow, and yields functional inequalities such as Wasserstein distance monotonicity \cite{McCannTopping}, and universal Poincar\'e and log-Sobolev inequalities \cite{HeinNaber}.  Furthermore, these functional inequalities can be used to characterize supersolutions to Ricci flow \cite{HaslhoferNaber, McCannTopping}.  Later, through a broad extension of the Bochner formula to functions on path space, Haslhofer-Naber gave a characterization of solutions to Ricci flow \cite{HaslhoferNaber} in terms of universal functional inequalities.  In this paper we extend this circle of ideas to the setting of \emph{generalized Ricci flow}.   A one-parameter family of metrics and two-forms $(g_t, b_t)$ is a solution of generalized Ricci flow \cite{Streetsexpent} if
\begin{align*}
\dt g =&\ - 2 \Rc + \tfrac{1}{2} H^2, \qquad \dt b = - d^*_g H, \qquad H = H_0 + db,
\end{align*}
where $d H_0 = 0$ and 
\begin{align*}
H^2(X,Y) = \IP{X \hook H, Y \hook H}.
\end{align*}
At times we will refer equivalently to the associated pair $(g_t, H_t)$ as a solution to generalized Ricci flow.  It is natural to express this equation using the curvature of the unique metric connection with torsion $H$, referred to as a \emph{Bismut connection}.  If we let $D$ denote the Levi-Civita connection, the relevant Bismut connection is then
\begin{align*}
\N := D + \tfrac{1}{2} g^{-1} H, \qquad \Rc^{\N} = \Rc - \tfrac{1}{4} H^2 - \tfrac{1}{2} d^*_g H.
\end{align*}
It follows that the generalized Ricci flow can be expressed as
\begin{align*}
\dt \left(g - b \right) = - 2 \Rc^{\N},
\end{align*}
where $\Rc^{\N}$ is the Ricci tensor of the Bismut connection.  The flow equation arises naturally as renormalization group flow \cite{Polchinski}, and arises naturally from considerations in complex geometry \cite{st-geom, PCF, GKRF} and generalized geometry \cite{GF19, StreetsTdual}.  We refer to \cite{GRFbook} for further background on generalized Ricci flow.

As solutions to generalized Ricci flow are supersolutions to Ricci flow, the results on Ricci flow supersolutions mentioned above immediately apply without changes.  However, by using the explicit geometric structure of generalized Ricci flow we obtain sharper results.  First we show universal Poincar\'e and log-Sobolev inequalities along solutions to generalized Ricci flow, extending the result of \cite{HeinNaber}.  It is possible to use these inequalities to give characterizations of supersolutions to generalized Ricci flow, although we do not carry this out here.  To state the result we record some notation: given $(M^n, g_t, H_t)$ a generalized Ricci flow on $M \times [0,T]$, for $(x_0,0)\in M \times [0,T]$ let $(s,y)\mapsto p_{T,s}(x_0,y)$ denote the conjugate heat kernel (see Definition \ref{defn: heat}), and let
\begin{align*}
d\nu_s^{x_0}=p_{T,s}(x_0,y)\, dV_{g(s)}.
\end{align*}
Throughout we adopt the convention that by a solution to generalized Ricci flow we mean a smooth solution where each time slice is complete with bounded geometry.

\begin{thm} \label{t:PLS} Let $(M^n, g_t, H_t)$ be a solution to generalized Ricci flow defined on $M \times [0, T]$.  Fix $x_0 \in M$ and $s \in [0, T)$. Then:
\begin{enumerate}
\item For all $\phi \in C_0^{\infty}(M)$ with $\int \phi d \nu_s^{x_0} = 0$, one has
\begin{gather*}
\int \phi^2 d \nu_s^{x_0} \leq 2 (T - s) \int \brs{\N \phi}^2 d \nu_s^{x_0},
\end{gather*}
with equality if and only if either $\phi \equiv 0$, or $(M,g_t, H_t) \cong (M',g'_t, H'_t)\times (\R,dz^2,0)$ for all $t\in[s,T]$ with $z(x_0) = 0$ and $\phi(x)=\lambda z$ for some constant $\lambda\in\R^*$.
\item For all $\phi \in C_0^{\infty}(M)$ with $\int \phi^2 d \nu_s^{x_0} = 1$, one has
\begin{gather*}
\int \phi^2 \log \phi^2 d \nu_s^{x_0} \leq 4 (T - s) \int \brs{\N \phi}^2 d \nu_s^{x_0},
\end{gather*}
with equality if and only if either $\phi \equiv 1$, or $(M,g_t, H_t) \cong (M',g'_t, H'_t)\times (\R,dz^2,0)$ for all $t\in[s,T]$ with $z(x_0) = 0$ and $\phi(x)=\exp( \gl z - 2 \gl^2 (T - s))$ for some constant $\lambda\in\R^*$.
\end{enumerate}
\end{thm}

Going further, we will show a generalization of the infinite-dimensional Bochner formula for the Malliavin gradient on path space along Ricci flow as in \cite{ACT08, HaslhoferNaber, Kennedy}.  The starting point of these constructions is to define a connection on the frame bundle of the spacetime associated to a time-dependent Riemannian manifold, originally employed in Hamilton's proof of the Harnack inequality for Ricci flow \cite{HamHarnack}.  It turns out that it is possible to incorporate the two-form potential $b_t$ into this construction in a way that fits very naturally with the generalized Ricci flow equation.  For a family $(M^n, g_t, b_t)$ defined for $t \in I$, we define a connection $\N$ on $\pi^* TM \to M \times I$ which extends the given action of $\N$ via
\begin{align*}
\N_t Y = \del_t Y + \tfrac{1}{2} \del_t \left(g_t - b_t \right)(Y, \cdot)^{\sharp_{g_t}}.
\end{align*}
This operator admits a key Bochner formula, which is central to our constructions.  In particular, given $(g_t, H_t = H_0 + d b_t)$ a general one-parameter family, and $u$ a solution of the time-dependent heat equation, one has that (Proposition \ref{p:gradientev})
\begin{align*}
\N_t \grad_{g_t} u =&\ \gD \grad_{g_t} u - \left( \Rc^{\N} + \tfrac{1}{2} \del_t \left( g_t - b_t \right) \right)\left( \grad_{g_t} u, \cdot \right)^{\sharp_{g_t}}.
\end{align*}
Thus, along a solution to generalized Ricci flow, the gradient of a solution to the heat equation itself satisfies a pure heat equation using the adapted derivative $\N$.  The main goal is to give an extension of the Bochner formula above to path space.  In \S \ref{s:Bochnerpath} we use the connection $\N$ on spacetime defined above together with the antidevelopment map to give the Eels-Elworthy-Malliavin construction of Brownian motion in this setting.  This in turn gives a notion of parallel gradient for martingales.  We then prove a formula on the evolution of parallel gradients of martingales which generalizes the Bochner identity above:
\begin{align*}
d(\nabla_\sigma^\perp F_\tau)=\IP{\nabla_\tau^\perp\nabla_\sigma^\perp F_\tau,\, dW_\tau}
+({\Rc^\nabla}+\frac12{\partial_t (g-b)})_\tau(\nabla_\tau^\perp F_\tau)\mathds{1}_{[\sigma,T]}(\tau)\, d\tau+\nabla_\sigma^\perp F_\sigma\delta_\sigma(\tau)\, d\tau.
\end{align*}
This is a generalization of the Bochner formula described above (cf. Corollary \ref{c:Bochnerreduction}), which occurs as the case where $F$ is a one-point cylinder function.

The path-space Bochner formula above can be used to give many equivalent characterizations of generalized Ricci flow.  First, in Theorem \ref{theorem: char} we give equivalent characterizations in terms of Bochner inequalities on path space.  Next, in Theorem \ref{thm: grad} we show equivalence with universal estimates on the norm and square norm of gradients of martingales.  We note that the adapted geometry on path space also determines an Ornstein-Uhlenbeck operator by composing the Malliavin gradient with its adjoint.  We show equivalence with universal Poincar\'e and log-Sobolev inequalities for this operator on path space, extending the inequalities of Theorem \ref{t:PLS}.  The precise definitions of the objects in the theorem below appear in \S \ref{s:Bochnerpath}.

\begin{thm12}
	For an evolving family of manifolds $(M,g_t,H_t)_{t\in[0,T]}$, the following are equivalent:
	\begin{enumerate}
		\item\label{ithmmain3: 1} The generalized Ricci flow 
		\begin{align*}
		\partial_t(g-b)=-2\Rc^\nabla
		\end{align*}
		is satisfied.
		\item\label{ithmmain3: 2}
		For any $0\leq\sigma\leq T'\leq T$ and any $F\in \cyl$, we have the estimate
		\begin{align*}
	\mathbb E_{(x,T')}[	|\nabla_\sigma^\perp F_{\sigma}|^2]+2 \int_0^{T'} \mathbb E_{(x,T')}[|\nabla_\tau^\perp \nabla_\sigma^\perp F_\tau|^2]\, d\tau \leq \mathbb E_{(x,T')}[|\nabla_\sigma^\perp F|^2] 
		\end{align*}
		 for all $ x\in M$.
			\item\label{ithmmain3: 6} For any $0\leq\tau_1\leq \tau_2\leq T'\leq T$ the Ornstein-Uhlenbeck operator $\LL_{(\tau_1,\tau_2)}$ on parabolic path space $L^2(P_{T'}\MM)$ satisfies the Poincar\'e inequality
		\begin{align*}
		\E_{(x,T')}[(F_{\tau_2}-F_{\tau_1})^2]\leq 2\E_{(x,T')}[F\, \LL_{(\tau_1,\tau_2)}F]
		\end{align*}
		for all $x\in M$.
		\item\label{ithmmain3: 5} For any $0\leq\tau_1\leq \tau_2\leq T'\leq T$ the Ornstein-Uhlenbeck operator $\LL_{(\tau_1,\tau_2)}$ on parabolic path space $L^2(P_{T'}\MM)$ satisfies the log-Sobolev inequality
		\begin{align*}
		&\mathbb E_{(x,T')}[(F^2)_{\tau_2}\log((F^2)_{\tau_2})-(F^2)_{\tau_1}\log((F^2)_{\tau_1})]
		\leq4\mathbb E_{(x,T')}[F\, \LL_{(\tau_1,\tau_2)}F]
		\end{align*}
		for all $x\in M$.
	\end{enumerate}
Moreover, if one of the conditions \eqref{ithmmain3: 1}-\eqref{ithmmain3: 5} is satisfied, we have:
\begin{enumerate}
	\item[(3a)]{\label{ithmmain3: 3a}}
	For any $0\leq T'\leq T$, $F\in \cyl$, we have the Poincar\'e Hessian estimate
	\begin{align*}
	\mathbb E_{(x,T')}[(F-\mathbb E_{(x,T')}[F])^2]+4\int_0^{T'}\int_0^{T'}\mathbb E_{(x,T')}[	|\nabla_\tau^\perp \nabla_\sigma^\perp F_{\tau}|^2]\, d\sigma\, d\tau\leq 2\int_0^{T'}\mathbb E_{(x,T')}[|\nabla_\sigma^\perp F|^2]\, d\sigma
	\end{align*}
	for all $x\in M$.
	\item[(4a)] {\label{ithmmain3: 4a}}For any $0\leq T'\leq T$, $F\in \cyl$, we have the log-Sobolev Hessian estimate
	\begin{align*}
	&\mathbb E_{(x,T')}[F^2\log(F^2)]-\mathbb E_{(x,T')}[F^2]\log(\mathbb E_{(x,T')}[F^2])\\
	&+2\int_0^{T'}\int_0^{T'}\mathbb E_{(x,T')}[(F^2)_\tau |\nabla_\tau^\perp\nabla_\sigma^\perp \log((F^2)_\tau)|^2]\, d\sigma\, d\tau
	\leq 4\int_0^{T'}\mathbb E_{(x,T')}[|\nabla_\sigma^\perp F|^2]\, d\sigma
	\end{align*}
	for all $x\in M$.
\end{enumerate}
\end{thm12}

\textbf{Acknowledgements:} The first named author gratefully acknowledge support by the German Research Foundation through the Hausdorff Center for Mathematics and the Collaborative Research Center 1060. The second named author was supported by the NSF via DMS-1454854.  We thank Robert Haslhofer for helpful conversations.

\section{Universal Poincar\'e and log-Sobolev inequalities along generalized Ricci flow}
\subsection{Conventions}

Before we begin we explicitly clarify our notational conventions and some elementary facts.  Given data $(M^n, g, H)$ of a smooth Riemannian manifold and closed three-form $H$, we let $D$ denote the Levi-Civita connection of $g$ and $\N = D + \tfrac{1}{2} g^{-1} H$ denote the Bismut connection, as explained in the introduction.  We almost exclusively work with $\N$, although in some proofs $D$ makes an appearance.  The connection $\N$ induces connections on all tensor bundles, and furthermore we define a Laplace operator
\begin{align*}
\gD = \tr_g \N \N
\end{align*}
A fundamental point is that the Laplacian acting on functions is the same as the usual Levi-Civita Laplacian, although importantly this is no longer the case for the Laplacian acting on other tensor bundles, in particular acting on $1$-forms and vector fields.  Furthermore, we will typically deal with a one-parameter family $(g_t, H_t = H_0 + db_t)$ of Riemannian metrics and closed three-forms.  Often we will simply describe this as $(g_t, b_t)$, with the background fixed choice of $H_0$ not stated.

\subsection{Heat operators}

\begin{defn}
The heat operator and conjugate heat operator along a solution to generalized Ricci flow are defined by
\begin{align*}
\square:=&\ \partial_t-\Delta\\
\square^*:=&\ -\partial_t-\Delta+R - \frac{1}{4} \brs{H}^2.
\end{align*}
\end{defn}

\begin{lemma}\label{lemma: dual}
Let $[t_1,t_2]\subset [0,T]$. Let $u,v\colon M\times [t_1,t_2]\to\R$ be smooth functions with compact support in $M$. Then
\begin{align*}
\int_{t_1}^{t_2}\int_M(\square u)v-(\square^* v)u\, d V\, dt=\left[\int_M uv\, d V\right]_{t_1}^{t_2}.
\end{align*}
\end{lemma}

\begin{defn}\label{defn: heat}
For $x,y\in M$ and $s<t\in[0,T]$, we let $p_{t,s}(x,y)$ denote the heat kernel based at $(s,y)$, i.e. the unique minimal positive solution to the equations
\begin{align*}
\square_{t,x}p_{t,s}(x,y) =&\ 0,\\
\lim_{t\downarrow s}p_{t,s}(x,y)=&\ \delta_y(x).
\end{align*}
\end{defn}

Observe that by the duality in Lemma \ref{lemma: dual}, the heat kernel $p_{t,s}(x,y)$ equivalently solves the 
conjugate heat equation based at $(t,x)$ in $(s,y)$, i.e.
\begin{align*}
\square_{s,y}^*p_{t,s}(x,y)=&\ 0\\
\lim_{s\uparrow t}p_{t,s}(x,y)=&\ \delta_x(y).
\end{align*}
Consequently, $p_{t,s}(x,y)$ is mass-preserving in $y$ with respect to $dV_{g(s)}$ and
\begin{align*}
\int_M p_{t,s}(x,y)\, d V_{g(s)}(y)=1
\end{align*}
for all $s<t$ and $x\in M$. Moreover, the uniqueness implies the propagator property
\begin{align*}
p_{t,r}(x,z)=\int_M p_{t,s}(x,y)p_{s,r}(y,z)\, dm_s(y)
\end{align*}
for all $r<s<t$ and $x,z\in M$.  With this we can define the heat flow and the conjugate heat flow of a function $u\in C_0^\infty(M)$.

\begin{defn} Let $u,v\in C_0^\infty(M)$ and $\mu\in\cP(M)$. For $s \leq t \in [0, T]$, let $(t,x)\mapsto P_{t,s} u(x)$ denote the heat flow, i.e.
\begin{align*}
(P_{t,s} u)(x) = \int_M p_{t,s}(x,y)u(y) \, dV_{g(s)}(y).
\end{align*}
For $t \geq s \in [0, T]$, let $(s,y)\mapsto P_{t,s}^* v(y)$ denote the conjugate heat flow, i.e.
\begin{align*}
(P_{t,s}^* v)(y) = \int_M p_{t,s}(x,y)v(x) \, dV_{g(t)}(x).
\end{align*}
\end{defn}
In other words, $(t,x)\mapsto P_{t,s}u(x)$ solves the (forward) heat equation from time $s$ with initial condition $u$ to time $t$, whereas $(s,y)\mapsto P_{t,s}^* v(y)$ solves the (backward) conjugate heat equation from time $t$ with terminal condition $v$ to time $s$.  Lastly we record a useful identidity for the heat flow $P_{t,s}$.

\begin{lemma}\label{lemma: operator}
Let $t\in[0,T]$.
For any family of smooth functions $U_s$ parametrized by $s\in(0,t)$,
\begin{align*}
\frac{d}{ds}P_{t,s} U_s=P_{t,s}\square_sU_s. 
\end{align*}
\end{lemma}
\begin{proof}
By the definition of $P_{t,s}$ we have for every $u\in C^\infty_0(M)$ 
\begin{align*}
\frac{d}{ds} P_{t,s}u=-P_{t,s}\Delta_{g(s)}u.
\end{align*}
The claim follows then from the Leibniz rule.
\end{proof}

\subsection{The parabolic Bochner formula}

A fundamental observation for Ricci flow equation is the gradient bound for solutions to the time-dependent heat equation.  In fact, this observation extends to supersolutions of the Ricci flow, and thus automatically applies to solutions of generalized Ricci flow.  The lemma below records the parabolic Bochner formulas for spacetime functions along a solution to generalized Ricci flow, which interestingly is expressed naturally in terms of the Hessian with respect to the Bismut connection.

\begin{lemma} \label{l:PBF} Let $(M^{n}, g_t, H_t)$ denote a solution to generalized Ricci flow.  Then the following hold:
\begin{enumerate}
\item Given $u$ a spacetime function,
\begin{align*}
\square \tfrac{1}{2} \brs{\N u}^2 =&\ - \brs{\N \N u}^2 + \IP{\N u, \N \square u}.
\end{align*}
\item Fix $u$ a spacetime function and $\varphi : \mathbb R \to \mathbb R$.  Setting $U = \varphi(u)$, we have
\begin{align*}
\square U = \varphi' \square u - \varphi'' \brs{\N u}^2.
\end{align*}
\item Fix $u$ a spacetime function and $\psi : \mathbb R \to \mathbb R$.  Setting $U = \psi(u) \brs{\N u}^2$ we have
\begin{align*}
\square U =&\  -2 \psi(u) \left( \brs{\N  \N  u}^2 + \IP{\N u, \N \square u} \right) - 4 \psi'(u) \IP{\N^2 u, \N u \otimes \N u}\\
&\ - \psi''(u) \brs{\N u}^4 + \psi'(u) \square u.
\end{align*}
\end{enumerate}
\begin{proof} For item (1) we apply the usual Bochner formula to obtain
\begin{align*}
\square \tfrac{1}{2} \brs{\N u}^2 =&\ \IP{\N \Delta u, \N u} + \IP{\Rc - \tfrac{1}{4} H^2, \N u \otimes \N u} - \brs{\N^2 u}^2 - \IP{\gD \N u, \N u}\\
=&\ - \brs{D^2 u}^2 - \tfrac{1}{4} \brs{\N u \hook H}^2 + \IP{\N u, \N \square u}.
\end{align*}
We furthermore observe that, since
\begin{align*}
\N = D + \tfrac{1}{2} H g^{-1},
\end{align*}
it follows that
\begin{align*}
\N \N u =&\ D D u + \frac{1}{2} \N u \hook H.
\end{align*}
And hence, since $D D u$ is symmetric and $\N u \hook H$ is skew-symmetric, it follows that
\begin{align*}
\brs{\N \N u}^2 = \brs{D D u}^2 + \tfrac{1}{4} \brs{\N u \hook H}^2,
\end{align*}
yielding item (1).  For item (2) we compute
\begin{align*}
\square \varphi(u) =&\ \varphi'(u) \frac{\del u}{\del t} - \divg \left( \varphi'(u) \N u \right) = \varphi'(u) \square u - \varphi''(u) \brs{\N u}^2.
\end{align*}
For item (3) we compute
\begin{align*}
\square U =&\ \psi'(u) \frac{\del u}{\del t} \brs{\N u}^2 + \psi(u) \dt \brs{\N u}^2 - \divg \left( \psi'(u) \N u \brs{\N u}^2 + \psi(u) \N \brs{\N u}^2 \right)\\
=&\ \psi'(u) \square u + \psi(u) \square \brs{\N u}^2 - \psi''(u) \brs{\N u}^4 - 4 \psi'(u) \IP{\N^2 u, \N u \otimes \N u}\\
=&\ -2 \psi(u) \left( \brs{\N  \N  u}^2 + \IP{\N u, \N \square u} \right) - 4 \psi'(u) \IP{\N^2 u, \N u \otimes \N u}\\
&\ - \psi''(u) \brs{\N u}^4 + \psi'(u) \square u.
\end{align*}
\end{proof}
\end{lemma}

Using this lemma we give two useful intertwining relations of the heat flow, in particular generalizing the $L^2$-gradient estimate in the sense of Bakry-\'Emery.
\begin{prop}\label{prop: gradest} Let $(M^{n}, g_t, H_t)$ denote a solution to generalized Ricci flow.
\begin{enumerate}
\item For $u\in C^\infty_0(M)$ it holds
\begin{align*}
\brs{\N P_{t,s}u}_{g(t)}^2=P_{t,s}(\brs{\N u}_{g(s)}^2)
-2\int_s^t P_{t,r}\left(\brs{\N \N P_{r,s}u}_{g(r)}^2\right)\, dr.
\end{align*}
\item For $u\in C^\infty_0(M)$ with $u\geq0$ it holds
\begin{align*}
\frac{\brs{\N P_{t,s}u}_{g(t)}^2}{P_{t,s}u}=P_{t,s}\left(\frac{\brs{\N u}_{g(s)}^2}{u}\right)
-2\int_s^t P_{t,r}\left(P_{r,s}u\brs{\N \N \log P_{r,s}u}_{g(r)}^2\right)\, dr.
\end{align*}
\end{enumerate}
\end{prop}

\begin{proof}
By Lemma \ref{lemma: operator} and Lemma \ref{l:PBF} (3) with $\psi=1$ we have
\begin{align*}
\int_s^t \frac{d}{dr} P_{t,r}( \brs{\N P_{r,s} u}_{g(r)}^2) dr =& \int_s^t P_{t,r} \square_r \brs{\N P_{r,s} u}_{g(r)}^2 dr - 2\int_s^t P_{t,r} \brs{\N^2 P_{r,s} u}^2_{g(r)} dr.
\end{align*}
On the other hand, by the fundamental theorem of calculus
\begin{align*}
\int_s^t \frac{d}{dr} P_{t,r}( \brs{\N P_{r,s} u}_{g(r)}^2)dr = \brs{\N P_{t,s} u}_{g(t)}^2- P_{t,s} (\brs{\N u}_{g(s)}^2).
\end{align*}
Combining the two last equations yields item (1).

In order to show item (2) we derive with Lemma \ref{lemma: operator} and Lemma \ref{l:PBF} (3) with the choice $\psi(u)=\frac1u$,
\begin{align*}
&\int_s^t \frac{d}{dr} P_{t,r}\left(\frac{ \brs{\N P_{r,s} u}_{g(r)}^2}{P_{r,s} u}\right) dr
= \int_s^t P_{t,r} \square_r\left(\frac{ \brs{\N P_{r,s} u}_{g(r)}^2}{P_{r,s}u}\right) dr\\
& \qquad =\ - 2\int_s^t P_{t,r} \left( \frac{\brs{\N \N P_{r,s} u}^2_{g(r)}}{P_{r,s}u} 
-2\frac{\IP{\N^2P_{r,s}u,\nabla P_{r,s}u\otimes\nabla P_{r,s}u}}{(P_{r,s}u)^2} 
+\frac{\brs{\N P_{r,s}u}_{g(r)}^4}{(P_{r,s}u)^3}\right) dr\\
& \qquad =\ - 2\int_s^t P_{t,r} \left(P_{r,s}u \brs{\N \N \log P_{r,s} u}^2_{g(r)} \right) dr.
\end{align*}
Again, the fundamental theorem of calculus yields the claim.
\end{proof}

\subsection{Proof of Theorem \ref{t:PLS}}

We end this section with the proof of Theorem \ref{t:PLS}.  We apply the intertwining relations from Proposition \ref{prop: gradest} for certain test functions.  We first prove a splitting result for Bismut connection.

\begin{prop} \label{p:splitting} Let $(M^n, g, H)$ be a smooth Riemannian manifold with $H$ a closed three-form.  Suppose there exists a closed nonvanishing $1$-form $\ga$ such that $\N \ga \equiv 0$.  Then the universal cover $\pi : \til{M} \to M$ splits as $\til{M} = M' \times \mathbb R$, the metric splits $\pi^* g = g' + \pi^* \ga \otimes \pi^* \ga$, where $g'$ is a metric on $M'$, and lastly $\pi^* H = \pi_{M'}^* H'$, where $H'$ is a closed three-form on $M'$.
\begin{proof} Using the definition of $\N $ we observe that
\begin{align*}
0 \equiv \N \ga = D \ga + \tfrac{1}{2} \ga^{\sharp} \hook H.
\end{align*}
Since $\ga$ is closed we have that $D \ga$ is symmetric, whilst the final term is skew-symmetric, thus the two terms on the right hand side above vanish individually.  In particular $\ga$ is parallel with respect to the Levi-Civita connection, and the metric splitting of the universal cover is a standard consequence of the deRham decomposition theorem.  

To show the splitting property of $H$ we let $z$ denote a coordinate on the $\mathbb R$-factor of $\til{M}$, and let $A,B,C \in TM'$.  Note that by construction $\pi^* \ga$ is a nonzero multiple of $\frac{\del}{\del z}$ and thus $\frac{\del}{\del z} \hook \pi^* H \equiv 0$.  Using this and that $\pi^* H$ is closed we furthermore obtain
\begin{align*}
0 =&\ d \pi^* H \left( \frac{\del}{\del z}, A, B, C \right) = D_{\frac{\del}{\del z}} \pi^* H(A,B,C).
\end{align*}
Thus $\pi^* H$ is parallel along $\frac{\del}{\del z}$, and it follows that $\pi^* H = \pi^*_{M'} H'$, with $dH' = 0$, as claimed.
\end{proof}
\end{prop}

\begin{proof}[Proof of Theorem \ref{t:PLS} (1)] Let $\varphi(x) = x^2$. Let $u\in C^\infty_0(M)$ and recall that
\begin{align*}
d\nu_s^{x_0}= p_{T,s}(x_0,\cdot)\, dV_{g(s)}.
\end{align*}
We note
\begin{align*}
- \int_s^T \frac{d}{dt} P_{T,t} ( \varphi( P_{t,s} u))(x_0) \, dt =&\ P_{T,s} (\varphi( P_{s,s} u))(x_0) - P_{T,T}(\varphi( P_{T,s} u))(x_0)\\
=&\ \int_M \varphi(u)(y) p_{T,s}(x_0,y) \, dV_{g(s)}(y) - \varphi \left( \int_M u(y) p_{T,s}(x_0,y) \, dV_{g(s)}(y) \right)\\
=&\ \int_M u^2 \, d\nu_s^{x_0} - \left( \int_M u \, d \nu_s^{x_0} \right)^2.
\end{align*}
But on the other hand using Lemma \ref{lemma: operator} and Lemma \ref{l:PBF} (2) we obtain
\begin{align*}
- \int_s^T \frac{d}{dt} P_{T,t} ( \varphi( P_{t,s} u))(x_0) \, dt =&\ - \int_s^T P_{T,t} \square_t (\varphi (P_{t,s} u))(x_0) \, dt\\
=&\ 2\int_s^T P_{T,t} ( \brs{\N P_{t,s} u}^2_{g(t)})(x_0)\,  dt.
\end{align*}
Applying Proposition \ref{prop: gradest} (1) to the last term we get
\begin{align*}
- \int_s^T \frac{d}{dt} P_{T,t} ( \varphi( P_{t,s} u))(x_0)\,  dt=&\ 2\int_s^T P_{T,t}\left [ P_{t,s}(\brs{\N u}^2_{g(s)})(x_0)-2\int_s^tP_{t,r}(\brs{\N \N P_{r,s}u}^2)(x_0)\, dr\right] \, dt
\end{align*}
Combining these equations yields
\begin{align*}
\int_M u^2 & \, d\nu_s^{x_0} - \left( \int_M u \, d \nu_s^{x_0}\right)^2\\
=&\ 2 \int_s^T P_{T,t} P_{t,s} \brs{\N u}_{g(s)}^2(x_0) \, dt - 4 \int_s^T P_{T,t} \left[ \int_s^t P_{t,r} \left( \brs{\N \N P_{r,s} u}^2_{g(r)}(x_0) \right)\,  dr \right] \, dt\\
=&\ 2 \int_s^T P_{T,s} \brs{\N u}_{g(s)}^2(x_0) dt - 4 \int_s^T P_{0,t} \left[ \int_s^t P_{t,r} \left( \brs{\N \N P_{r,s} u}^2_{g(r)}(x_0) \right) \, dr \right] \, dt\\
=&\ 2 (T - s) \int_M \brs{\N u}_{g(s)}^2(y) p_{T,s}(x_0,y)\,  dV_{g(s)}(y) -4 \int_s^T P_{T,t} \left[ \int_s^t P_{t,r} \left( \brs{\N \N P_{r,s} u}^2_{g(r)}(x_0) \right) \, dr \right] \, dt\\
=&\ 2 (T - s) \int_M \brs{\N u}_{g(s)}^2 \, d\nu_s^{x_0} - 4 \int_s^T P_{T,t} \left[ \int_s^t P_{t,r} \left( \brs{\N \N P_{r,s} u}^2_{g(r)}(x_0) \right) \, dr \right] \, dt\, .
\end{align*}
This implies the Poincar\'e inequality of item (1). 

Equality occurs if and only if 
\begin{align*}
 \int_s^T P_{T,t} \left[ \int_s^t P_{t,r} \left( \brs{\N \N P_{r,s} u}^2_{g(r)}(x_0) \right) \, dr \right] \, dt=0.
\end{align*}
Due to the maximum principle of $P_{t,r}$ and $P_{T,t}$ we have that $\brs{\N\N P_{r,s}u}_{g(r)}$ needs to vanish at time $r=s$, i.e. $\brs{\N\N  u}_{g(s)}=0$.  It follows from Proposition \ref{p:splitting} that the data at time $s$ splits as claimed.  From uniqueness of solutions to generalized Ricci flow in the class of solutions with bounded geometry (cf. \cite{ChenZhu} which extends to generalized Ricci flow), it follows that the solution splits for all times, as claimed.
\end{proof}

\begin{proof}[Proof of Theorem \ref{t:PLS} (2)] Let $\varphi(x)=x\log x$ and $u\in C^\infty_0(M)$ with $u\geq0$.
With this we obtain from Proposition \ref{prop: gradest} (2) that
\begin{align*}
\int_M u\log u\, d\nu_s^{x_0}-\int_M u\, d\nu_s^{x_0}\log\left(\int_M u\, d\nu_s^{x_0}\right)
=\int_s^T P_{T,t}\left(\frac{\brs{\N P_{t,s}u}_{g(t)}^2}{P_{t,s}u}\right)(x_0)\, dt.
\end{align*}
Applying Proposition \ref{prop: gradest} to the right hand side we get
\begin{align*}
& \int_M u\log u\, d\nu_s^{x_0}-\int_M u\, d\nu_s^{x_0}\log\left(\int_M u\, d\nu_s^{x_0}\right)\\
&\ \qquad = \int_s^T P_{T,t}\left[P_{t,s}\left(\frac{\brs{\N u}_{g(s)}^2}{u}\right)-2\int_s^tP_{t,r}(P_{r,s}u \brs{\N \N \log P_{r,s}u}^2)\, dr\right](x_0)\, dt\\
&\ \qquad = (T - s)P_{T,s}\left(\frac{\brs{\N u}_{g(s)}^2}{u}\right)-2\int_s^T P_{T,t}\left[\int_s^tP_{t,r}(P_{r,s}u \brs{\N \N \log P_{r,s}u}^2)\, dr\right](x_0)\, dt\\
&\ \qquad = (T - s) \int_M\frac{\brs{\N u}_{g(s)}^2}{u}\, d\nu_s^{x_0}-2\int_s^T P_{T,t}\left[\int_s^tP_{t,r}(P_{r,s}u \brs{\N \N \log P_{r,s}u}^2)\, dr\right](x_0)\, dt.
\end{align*}
Then if $\int_M u\, d\nu_s^{x_0}=1$ we set $\phi=\sqrt{u}$ and obtain item (2).  The case of equality is treated the same as in item (1).
\end{proof}

\section{Twisted parallel transport and Frame bundle formalism}

In this section we define a connection on a spacetime adapted to a solution of generalized Ricci flow.  The key point is a Bochner formula for solutions to the time-dependent heat equation, Proposition \ref{p:gradientev}, which lies at the heart of the path space constructions to follow.  We also use this connection to recast the time-dependent geometry on the frame bundle, which is necessary for the construction of the adapted Brownian motion et al.

\subsection{The twisted connection on spacetime}

Let $(M^n, g_t, b_t)$ be a one-parameter family of generalized metrics.  Let $\mathcal M = M \times I$ for some time interval $I$.  Let $\del_t$ denote the canonical vector field on $I$ lifted to $\mathcal M$.  We define a connection on the vector bundle $\pi^* TM \to \mathcal M$ which extends the action of $\N$ via
\begin{align*}
\N_t Y =&\ \del_t Y + \tfrac{1}{2} \del_t \left(g_t + b_t \right)(Y, \cdot)^{\sharp_{g_t}},
\end{align*}
This generalizes Hamilton's spacetime connection introduced in his derivation of the Harnack estimate \cite{HamHarnack}.  The term involving the time derivative of $g$ is natural to include as it renders the connection compatible with the time-dependent metric.  In fact one is free to add the action of an arbitrary skew-symmetric two-form as well and still preserve this property (cf. Lemma \ref{l:compatibility}).  As it turns out, the precise term $\tfrac{1}{2} \del_t b_t$ gives the connection $\N$ particularly favorable properties in the case of a solution to generalized Ricci flow.

\begin{lemma} \label{l:compatibility} The spacetime connection $\N$ is compatible with $g$.
\begin{proof}
This follows from
\begin{align*}
\frac{d}{dt} \brs{Y}^2_{g_t} =&\ \del_t g_t \left( Y, Y \right) + g_t \left( \del_t Y, Y \right) + g_t \left( Y, \del_t Y \right)\\
=&\ g_t \left( \del_t Y + \tfrac{1}{2} \del_t g_t(Y, \cdot)^{\sharp_{g_t}}, Y \right) + g_t \left( Y, \del_t Y + \tfrac{1}{2} \del_t g_t(Y, \cdot)^{\sharp_{g_t}}\right)\\
=&\ g_t \left( \del_t Y + \tfrac{1}{2} \del_t \left( g_t + b_t \right) (Y, \cdot)^{\sharp_{g_t}}, Y \right) + g_t \left( Y, \del_t Y + \tfrac{1}{2} \del_t \left( g_t + b_t \right) (Y, \cdot)^{\sharp_{g_t}}\right)\\
=&\ 2 g_t (\N_t Y, Y),
\end{align*}
where the third line follows since $b$ is skew-symmetric.
\end{proof}
\end{lemma}

\begin{lemma} \label{l:commutatorformula} Given $(M, g, H)$ as above one has
\begin{align*}
\gD \N u = \N \gD u + \Rc^{\N}( \N u, \cdot).
\end{align*}
\end{lemma}
\begin{proof}
We choose local coordinates and let $\gG$ denote the connection coefficients of $\N $.  We then compute
\begin{align*}
\left( \N  \N  \N  u \right)_{ijk}  =&\ \del_i \left( \N  \N u \right)_{jk} - \gG_{ij}^l \left(\N  \N u \right)_{lk} - \gG_{i k}^l \left(\N  \N u \right)_{jl}\\
=&\ \del_i \left( \del_j \del_k u - \gG_{jk}^p \del_p u \right) - \gG_{ij}^l \left( \del_l \del_k u - \gG_{lk}^p \del_p u \right) - \gG_{ik}^l \left( \del_j \del_l u - \gG_{jl}^p \del_p u \right).
\end{align*}
It follows that
\begin{align*}
\left( \N  \N  \N  u \right)_{ijk}  & - \left( \N  \N  \N  u \right)_{jik}\\
=&\ \del_i \left( \del_j \del_k u - \gG_{jk}^p \del_p u \right) - \gG_{ij}^l \left( \del_l \del_k u - \gG_{lk}^p \del_p u \right) - \gG_{ik}^l \left( \del_j \del_l u - \gG_{jl}^p \del_p u \right)\\
&\ - \left( \del_j \left( \del_i \del_k u - \gG_{ik}^p \del_p u \right) - \gG_{ji}^l \left( \del_l \del_k u - \gG_{lk}^p \del_p u \right) - \gG_{jk}^l \left( \del_i \del_l u - \gG_{il}^p \del_p u \right) \right)\\
=&\ \del_p u \left( - \del_i \gG_{jk}^p + \del_j \gG_{ik}^p - \gG_{jk}^l \gG_{il}^p + \gG_{ik}^l \gG_{jl}^p \right) - H_{ij}^l (\N  \N  u)_{lk}\\
=&\ - (R^{\N})_{ijk}^p d_p u - H_{ij}^l (\N  \N  u)_{lk}.
\end{align*}
Also we have
\begin{align*}
\N _j \N _k u - \N _k \N _j u =&\ \left( \del_j \del_k u - \gG_{jk}^p \del_p u \right) - \left( \del_k \del_j u - \gG_{kj}^p \del_p u \right) = - H_{jk}^p \del_p u.
\end{align*}
Combining these we then have
\begin{align*}
\N _i \gD u =&\ g^{jk} \N _i \N _j \N _k u\\
=&\ g^{jk} \left( \N _j \N _i \N _k u - (R^{\N})_{ijk}^p d_p u - H_{ij}^l (\N  \N  u)_{lk} \right)\\
=&\ g^{jk} \left( \N _j \N _k \N _i u - \N _j (H_{ik}^p d_p u) - (R^{\N})_{ijk}^p d_p u - H_{ij}^l (\N  \N  u)_{lk} \right)\\
=&\ \gD \N_i u - (\Rc^{\N})_{i}^p d_p u - g^{jk} \left( \N _j (H_{ik}^p d_p u) + H_{ij}^l (\N  \N  u)_{lk} \right).
\end{align*}
Then we simplify
\begin{align*}
g^{jk} & \left( \N_j \left( H_{ik}^p d_p u \right) + H_{ij}^l (\N  \N  u)_{lk} \right)\\
=&\ g^{jk} \left( D_j H_{ik}^p - \tfrac{1}{2} H_{ji}^l H_{lk}^p - \tfrac{1}{2} H_{jk}^l H_{il}^p + \tfrac{1}{2} H_{jl}^p H_{ik}^l \right) d_p u + g^{jk} H_{ik}^p \N_j \N_p u + g^{jk} H_{ij}^l \N _l \N _k u\\
=&\ (d^*_g H_{i}^p) d_p u,
\end{align*}
where the last line follows using the skew-symmetry of $H$.  This finally yields, after rearranging,
\begin{align*}
\gD \N_i u =&\ \N_i \gD u + (\Rc^\N)_i^p d_p u + (d^*_g H)_i^p d_p u\\
=&\ \N_i \gD u + (\Rc - \tfrac{1}{4} H^2 - \tfrac{1}{2} d^*_g H)_i^p d_p u + (d^*_g H)_i^p d_p u\\
=&\ \N_i \gD u + g^{pq} \Rc^{\N}_{qi} d_p u,
\end{align*}
as claimed.
\end{proof}

\begin{prop} \label{p:gradientev} Let $(M^n, g_t, b_t)$ denote a time dependent family as above, and suppose $u$ solves $\square u = 0$.  Then
\begin{align*}
\N_t \grad_{g_t} u =&\ \gD \grad_{g_t} u - \left( \Rc^{\N} + \tfrac{1}{2} \del_t \left( g_t - b_t \right) \right)\left( \grad_{g_t} u, \cdot \right)^{\sharp_{g_t}}.
\end{align*}
\end{prop}

\begin{proof} Using the definitions above and Lemma \ref{l:commutatorformula} we have
\begin{align*}
\N_t \grad_{g_t} u =&\ \del_t \left( g_t^{-1} d u \right) + \tfrac{1}{2} \del_t \left(g_t + b_t \right) \left(\grad_{g_t} u, \cdot \right)^{\sharp_{g_t}}\\
=&\ \grad_{g_t} \gD_{g_t} u - \del_t g_t \left( \grad_{g_t} u, \cdot \right)^{\sharp_{g_t}} + \tfrac{1}{2} \del_t \left(g_t + b_t \right) \left(\grad_{g_t} u, \cdot \right)^{\sharp_{g_t}}\\
=&\ \gD \grad_{g_t} u - \left( \Rc^{\N} +\ \tfrac{1}{2} \del_t (g_t - b_t) \right) \left( \grad_{g_t} u, \cdot \right)^{\sharp_{g_t}}.
\end{align*}
\end{proof}

\subsection{Frame bundle formalism}

We next use the spacetime connection to define an adapted geometry on the frame bundle.  We refer to \cite{KN1} for general background.  Let $(M^n, g_t, b_t)$ be a time-dependent family as above.  Let $\mathcal M = M \times I$ for some time interval $I$.  We define an $O_n$-bundle $\pi : \mathcal F \to \mathcal M$, where the fibres $\FF_{(x,t)}$ are orthogonal maps $u: \mathbb R^n \to (T_xM, g_t)$.  Given a curve $\gg_t$ in $\mathcal M$, the horizontal lift is a curve $u_t \in \FF$ such that:
\begin{align*}
\pi \circ u_t = \gg_t, \qquad \N_{\dot{\gg}} (u_t v) = 0 \quad \mbox{ for all } v \in \mathbb R^n.
\end{align*}
By general theory, it follows that if we fix a point $u_0 \in \pi^{-1} \gg_0$, there exists a unique horizontal lift $u_t$ with initial condition $u_0$.  Furthermore, given $a X + b \del_t \in T_{(x,t)} \mathcal M$, and $u \in \FF_{(x,t)}$, there exists a unique horizontal lift $a X^* + b \del^*_t$ which satisfies
\begin{align*}
\pi_* (a X^* + b \del^*_t) = a X + b \del_t.
\end{align*}
Here $X^*$ is the usual horizontal lift of $X$ with respect to $\N$, and $\del_t^*$ the lift of $\del_t$ along the path which is constant in space.

The frame bundle $\FF$ comes equipped with certain canonical vector fields as well.  First, as above we have $\del_t^*$ which is the horizontal lift of $\del_t$.  Furthermore we define horizontal vector fields $\{E_i\}_{i=1}^n$ via
\begin{align*}
E_i(u) = (u e_i)^*,
\end{align*}
where $\{e_i\}_{i=1}^n$ is the standard basis for $\mathbb R^n$.  We furthermore define vertical vector fields
\begin{align*}
V_{ij}(u) = \left. \frac{d}{ds} \right|_{s=0} (u \exp (s A_{ij})), \qquad (A_{ij})_{kl} = (\gd_{ik} \gd_{jl} - \gd_{il} \gd_{jk}).
\end{align*}
We will perform some computations below in local coordinates.  To that end, given coordinates $\{x^i, t\}$ on $\mathcal M$ we canonically associate coordinates $(x^i, t, e_j^k)$ on $\FF$, where the functions $e_j^k$ are defined by
\begin{align*}
u e_j = e_j^k \frac{\del}{\del x^k}.
\end{align*}
We furthermore recall an identification between contravariant tensor fields on $\mathcal M$ and equivariant functions on $\FF$.  Given a smooth function $f : \mathcal M \to \mathbb R$ we set
\begin{align*}
\til{f} = f \circ \pi.
\end{align*}
Furthermore, given $\ga = \ga_i dx^i \in T^* \mathcal M$, we obtain $\til{\ga} : \mathcal F \to \mathbb R^n$ via
\begin{align*}
\til{\ga}_i(u) = \ga_{\pi(u)} (u e_i).
\end{align*}
This identification extends in an obvious way to any contravariant tensor.

\begin{lemma}\label{lma: commute} With the setup above, one has:
\begin{align*}
\til{Xf} =&\ X^* \til{f}, \qquad \til{\del_t f} = \del_t^* \til{f}.
\end{align*}
\end{lemma}

\begin{prop} Given local coordinates $\{x^i, t\}$, we can express
\begin{align*}
E_i =&\ e_i^k \left( \frac{\del}{\del x^k} - e_j^l \gG_{kl}^m \frac{\del}{\del e_j^m} \right)\\
V_{ij} =&\ e_j^k \frac{\del}{\del e_i^k} - e_i^k \frac{\del}{\del e_j^k}\\
\del_t^* =&\ \del_t - \tfrac{1}{2} \left( \til{\del_t g_{ik}} - \til{\del_t b}_{ik} \right) e_k^j \frac{\del}{\del e_i^j}.
\end{align*}
\begin{proof} The first two items are already shown in \cite{HamHarnack}.  To show the final item we first fix a frame $u_0 \in \FF$, set $\pi(u_0) = (x_0, t_0)$, then define a curve $\gg(t) = (x_0, t_0 + t)$, and let $u_t$ denote the horizontal lift of $\gg_t$.  Recall that by construction, the vector field $u_t e_i$ is parallel, thus we can compute using the definition of the spacetime connection,
\begin{align*}
0 =&\ \N_{\dt} (u_t e_i) = \N_{\dt} \left( e_i^j \frac{\del}{\del x^j} \right)\\
=&\ \frac{d}{dt} e_i^j \frac{\del}{\del x^j} + e_i^j \N_{\dt} \frac{\del}{\del x^j}\\
=&\ \left( \frac{d}{dt} e_i^j  + \tfrac{1}{2} \left( \til{\del_t g}_{ik} - \til{\del_t b}_{ik} \right) e_k^j \right) \frac{\del}{\del x^j}.
\end{align*}
It follows that
\begin{align*}
\left. \del_t^* \right|_{u_0} =&\ \left. \frac{d}{dt} \right|_{t=0} \left( x_0, t_0 + t, e_i^j(t) \right)\\
=&\ \del_t + \left. \frac{d}{dt} \right|_{t=0} e_i^j(t) \frac{\del}{\del e_i^j}\\
=&\ \del_t - \tfrac{1}{2} \left( \til{\del_t g}_{ik} - \til{\del_t b}_{ik} \right) e_k^j \frac{\del}{\del e_i^j},
\end{align*}
as claimed.
\end{proof}
\end{prop}

\begin{prop}\label{prop: commute} With the setup above, for a contravariant tensor $T$ one has
\begin{align*}
\til{\N_X T} =&\ X^* \til{T}, \qquad \til{\N_t T} = \del_t^* \til{T}\\
\N^2 T(u e_i, u e_j) =&\ E_i E_j \til{T}, \qquad \til{\gD T} = \sum_{i=1}^n E_i E_i \til{T}.
\end{align*}
\begin{proof} We prove these identities for $T = \ga \in T^*$, as the general case is analogous.  Fix a vector field $X$, local coordinates $\{x^i, t\}$, and express $X = X^i \frac{\del}{\del x^i}$.  To simplify notation we let $\gG$ denote the Christoffel symbols of $\N $.  We first express
\begin{align*}
\left( \N_X \ga \right)_i = X^j \left( \N_{\frac{\del}{\del x^j}} \ga \right)_i = X^j \left( \frac{\del \ga_i}{\del x^j} - \gG_{ji}^k \ga_k \right).
\end{align*}
On the other hand we can express
\begin{align*}
\left( \frac{\del}{\del x^i} \right)^* = \frac{\del}{\del x^i} - \gG_{i l}^k e_j^l \frac{\del}{\del e_j^k} 
\end{align*}
It follows that
\begin{align*}
X^* \til{\ga}_i =&\ X^m \left( \frac{\del}{\del x^m} - \gG_{m l}^k e_j^l \frac{\del}{\del e_j^k} \right) \left( \ga_k e_i^k \right)\\
=&\ X^m \left( \frac{\del \ga_k}{\del x^m} e_i^k - \gG_{ml}^k \ga_k e_i^l \right)\\
=&\ X^m \left( \frac{\del \ga_k}{\del x^m} - \gG_{mk}^l \ga_l\right) e_i^k\\
=&\ \left( \til{\N_X \ga} \right)_i,
\end{align*}
as claimed.  For the third claim we first compute
\begin{align*}
E_i E_j \til{\ga}_k =&\ \left( e_i^s \left( \frac{\del}{\del x^s} - e_a^t \gG_{st}^m \frac{\del}{\del e_a^m} \right) \right) \left( e_j^p \left( \frac{\del}{\del x^p} - e_b^d \gG_{pd}^q \frac{\del}{\del e_b^q} \right) \right) \left( e_k^r \ga_r \right)\\
=&\ \left( e_i^s \left( \frac{\del}{\del x^s} - e_a^t \gG_{st}^m \frac{\del}{\del e_a^m} \right) \right) \left( e_j^p \left( e_k^r \frac{\del \ga_r}{\del x^p} -e_k^d \gG_{pd}^r \ga_r \right) \right)\\
=&\ e_i^s \left( e_j^p\left( e_k^r \frac{\del^2 \ga_r}{\del x^p \del x^s} - e_k^d \frac{\del \gG_{pd}^r}{\del x^s} \ga_r - e_k^d \gG_{pd}^r \frac{\del \ga_r}{\del x^s} \right)  \right)\\
&\ - e_i^s e_j^t \gG_{st}^p \left( e_k^r \frac{\del \ga_r}{\del x^p} - e_k^l \gG_{pl}^r \ga_r \right) - e_i^s e_k^t \gG_{st}^r e_j^p \frac{\del \ga_r}{\del x^p} + e_i^s e_k^t \gG_{st}^d e_j^p \gG_{pd}^r \ga_r
\end{align*}
On the other hand we have
\begin{align*}
\left( \N^2 \ga \right)_{ijk} =&\ \del_i (\N \ga)_{jk} - \gG_{ij}^l (\N \ga)_{lk} - \gG_{ik}^l (\N \ga)_{jl}\\
=&\ \del_i \left( \del_j \ga_k - \gG_{jk}^p \ga_p \right) - \gG_{ij}^l \left( \del_l \ga_k - \gG_{lk}^p \ga_p \right) - \gG_{ik}^l \left(\del_j \ga_l - \gG_{jl}^p \ga_p \right)\\
=&\ \ga_{k,ij} - \gG_{jk,i}^p \ga_p - \gG_{jk}^p \ga_{p,i} - \gG_{ij}^l \ga_{k,l} + \gG_{ij}^l \gG_{lk}^p \ga_p - \gG_{ik}^l \ga_{l,j} + \gG_{ik}^l \gG_{jl}^p \ga_p.
\end{align*}
Comparing these two formulas gives the result.  The final claim follows by tracing the third formula over an orthonormal basis.
\end{proof}
\end{prop}

\section{A Bochner formula on path space} \label{s:Bochnerpath}
\subsection{Brownian motion and Ito's Lemma}

 Let $(M^n,g_t,b_t)_{t\in [0,T]}$ be a time-dependent family with spacetime connection as above.  In what follows we give the Eels-Elworthy-Malliavin construction of Brownian motion adapted to our setting.  This is a further generalization of the construction of \cite{ACT08, HaslhoferNaber}.  Let $(x,T')\in\mathcal M$. It will be convenient to work with the backward time $\tau:=T'-t$ and the convention that $\partial^*_\tau=-\partial^*_t$. Let us start with a smooth curve $\gamma_\tau=(x_\tau,T'-\tau)$ in $\mathcal M$ with $x_0=x$ and denote by $u_\tau$ its horizontal lift. The anti-development $(w_\tau)_\tau\subset \R^n$ is the given as the solution of the ordinary differential equation
\begin{align*}
\dot u_\tau=\partial^*_\tau+E_i(u_\tau)\dot w_\tau^i, \qquad w_0 = 0,
\end{align*}
which exists along $\gg$ by general theory.  This equivalent formulation of parallel transport motivates the following stochastic differential equation
\begin{equation} \begin{aligned}\label{eq: anti}
d U_\tau=&\ \partial^*_\tau\, d \tau + E_i(U_\tau)\circ\, dW_\tau^i,\\
U_0=&\ u.
\end{aligned}\end{equation}
Here, $(W_\tau)$ is a Brownian motion on $\R^n$ and $u$ is an initial frame at $(x,T')$. We use the convention that $(W_\tau)$ has $\Delta^{\R^n}$ as generator instead of $\frac12\Delta^{\R^n}$, i.e. the covariation satisfies $dW_\tau^idW_\tau^j=2\delta_{ij}d\tau$ and $\circ$ refers to the Stratonovich integration.  In this section we establish existence and uniqueness of \eqref{eq: anti} as well as a version of Ito's lemma.

\begin{prop}
	The stochastic differential equation \eqref{eq: anti} has a unique continuous solution $(U_\tau)_{\tau\in[0,T']}$ and satisfies $\pi_2(U_\tau)=T'-\tau$. Furthermore, given any $C^2$ function $\til f:\FF\to\R$ we have
	\begin{equation}\begin{aligned}\label{eq: Ito}
&d\til f(U_\tau)=E_i\til f(U_\tau)\, dW^i_\tau + \partial^*_\tau\til f(U_\tau)\, d\tau+E_iE_i\til f(U_\tau)\, d\tau.
\end{aligned}\end{equation}
\end{prop}
\begin{proof}
	We adapt the corresponding argument from \cite{HaslhoferNaber}.  First, we may embed the manifold $\FF$ into $\R^N$ for some $N$. Then $U_\tau$ satisfies \eqref{eq: anti} if and only if the coordinate functions $U_\tau^a$ satisfy 
	\begin{align*}
	d U_\tau^a=&\ (\partial_\tau^*)^a\, d \tau + E_i^a(U_\tau)\circ\, dW_\tau^i
	\end{align*}
	for all $a=1,\ldots,N$, see \cite[Prop 1.2.7]{Hsu}. Since each vector field $E_i$ is smooth and bounded since each time slice has bounded geometry, it follows from the standard theory for SDEs on Euclidean space that there is a unique solution on $[0,T']$, cf. \cite[Theorem 1.1.8]{Hsu} and that this solution actually stays in $\FF$, see \cite[Theorem 1.2.9]{Hsu}.
	
	In order to show \eqref{eq: Ito} we convert the Stratonovich integral in \eqref{eq: anti} into an Ito integral by dropping the $\circ$ and adding one half times the covariation of $E_i(U_\tau)$ and $W_\tau$:
	\begin{align*}
	d U_\tau^a=&\ (\partial_\tau^*)^a\, d\tau + E_i^a(U_\tau)\, dW_\tau^i + \frac12\, dE_i^a(U_\tau)\, dW_\tau^i.
	\end{align*}
	For the covariation term we compute, using Ito's lemma in Euclidean space, 
	\begin{align*}
	dE_i^a(U_\tau)\, dW_\tau^i= \frac{\partial}{\partial x^b}E_i^a(U_\tau)\, dU_\tau^b\, dW_\tau^i
	= 2\frac{\partial}{\partial x^b}E_i^a(U_\tau)E_i^b(U_\tau)\, d\tau.
	\end{align*}
	Here, we also used the fact that the covariation of a continuous process and a process of finite variation vanishes.
	Now, let $\til f\colon\FF\to\R$ be a $C^2$ function. Then, by Ito's lemma in Euclidean space,
	\begin{align*}
	d\til f(U_\tau)=&\ \frac{\partial}{\partial x^a}\til f(U_\tau)\, dU_\tau^a+\frac12\frac{\partial^2}{\partial x^a\partial x^b}\til f(U_\tau)\, dU_\tau^a\, dU_\tau^b\\
	=&\ \frac{\partial}{\partial x^a}\til f(U_\tau)(\partial_\tau^*)^a\, d \tau
	+\frac{\partial}{\partial x^a}\til f(U_\tau)E_i^a(U_\tau)\, dW_\tau^i
	+\frac{\partial}{\partial x^a}\til f(U_\tau)\frac{\partial}{\partial x^b}E_i^a(U_\tau)E_i^b(U_\tau)\, d\tau\\
	&\ +\frac{\partial^2}{\partial x^a\partial x^b}\til f(U_\tau)E_i^a(U_\tau)E_i^b(U_\tau)\, d\tau.
	\end{align*}
	Finally, since
	\begin{align*}
	&\frac{\partial}{\partial x^a}\til f(U_\tau)(\partial_\tau^*)^a=\partial^*_\tau \til f(U_\tau)\\
	&\frac{\partial}{\partial x^a}\til f(U_\tau)E_i^a(U_\tau)=E_i(U_\tau)\til f(U_\tau)\\
	&\frac{\partial}{\partial x^a}\til f(U_\tau)\frac{\partial}{\partial x^b}E_i^a(U_\tau)E_i^b(U_\tau)
	+\frac{\partial^2}{\partial x^a\partial x^b}\til f(U_\tau)E_i^a(U_\tau)E_i^b(U_\tau)
	=E_i(U_\tau)E_i(U_\tau)\til f(U_\tau),
	\end{align*}
	we obtain \eqref{eq: Ito}:
	\begin{align*}
	d\til f(U_\tau)=\partial_\tau^* \til f(U_\tau)\, d\tau+E_i(U_\tau)\til f(U_\tau)\, dW_\tau^i+
	E_i(U_\tau)E_i(U_\tau)\til f(U_\tau)\, d\tau.
	\end{align*}
	Lastly, note that with the choice $\til f=\pi_2$ we get $d\til f(U_\tau)=-\, d\tau$. Furthermore setting $\pi_2(U_0)=T'$ we get $\pi_2(U_\tau)=T'-\tau$.
\end{proof}

Let $P_0\R^n$ denote the Euclidean path space based at the origin, i.e. the space of all continuous curves $\{w_\tau|w_0=0\}_{\tau\in[0,T']}\subset\R^n$. We denote by $\Gamma_0$ the Wiener measure on $P_0\R^n$. The path space has a canonical filtration $\Sigma_\tau^{\R^n}$ generated by the evaluation maps $\{e_\sigma\colon P_0\R^n\to \R^n|e_\sigma(w)=w_\sigma,\sigma\leq\tau\}$. With the help of \eqref{eq: anti} we can transfer the notion of Wiener measure to the path space over $\FF$ and $\mathcal M$.
\begin{defn}
 Let $P_u\FF$ and $P_{(x,T')}\mathcal M$ be the space of continuous curves, $\{u_\tau|u_0=u,\pi_2(u_\tau)=T'-\tau\}_{\tau\in[0,T']}\subset\FF$ and $\{\gamma_\tau=(x_\tau,T'-\tau)|\gamma_0=(x,T')\}_{\tau\in[0,T']}$ respectively.	
\end{defn}

It will be convenient from time to time to work with the \emph{total path space} $P_{T'}\MM=\bigcup_{x\in M}P_{(x,T')}\MM$.

 \begin{defn}
 	Let $U\colon P_0\R^n\to P_u\FF$ solve \eqref{eq: anti} and let $\Pi\colon P_u\FF\to P_{(x,T')}\mathcal M$ defined by $\Pi(U)_\tau=\pi(U_\tau)$. 
 	\begin{enumerate}
 		\item  We call $\Gamma_u:=U_*(\Gamma_0)$ and $\Gamma_{(x,T')}:=\Pi_*\Gamma_u$ the \emph{Wiener measures} of horizontal Brownian motion on $\FF$ and Brownian motion on space-time $\mathcal M$ respectively.
 		\item The filtrations on $P_u\FF$ and $P_{(x,T')}\mathcal M$ are given by $\Sigma_\tau^{\mathcal M}:=(\Pi\circ U)_*\Sigma_\tau^{\R^n}$ and $\Sigma_\tau^{\FF}:=U_*\Sigma_\tau^{\R^n}$.
 		\item We call $\pi(U_\tau)=(X_\tau,T'-\tau)$ \emph{Brownian motion} on $\mathcal M$ based at $\pi(u)=(x,T')$. 
\item We call the family of isometries $\{S_\tau:=U_0U_\tau^{-1}\colon (T_{X_\tau}M,g_{T'-\tau})\to (T_xM,g_{T'})\}$ \emph{stochastic parallel transport} along the Brownian curve $X_\tau$.
 	\end{enumerate}
 	\end{defn}

\begin{prop}
 Let $w$ on $\mathcal M$ be a solution to the heat equation
 \begin{align*}
\square w = 0, \quad w|_s=f,
 \end{align*}
 where $f\in C^\infty(M)$ and $s\in[0,T']$. Then
\begin{align*}
w(x,T)=\mathbb E_{(x,T')}[f(X_{T'-s})].
\end{align*}
\end{prop}
\begin{proof}
 We consider the lift $\til w(U_\tau)$ and obtain by \eqref{eq: Ito}
 \begin{align*}
 d\til w(U_\tau)=E_i\til w(U_\tau)\, dW_\tau^i+\partial^*_\tau\til w(U_\tau)\, d\tau+E_iE_i\til w(U_\tau)\, d\tau.
 \end{align*}
 Since $w$ solves the heat equation,  by virtue of Lemma \ref{lma: commute} and Proposition \ref{prop: commute} the last two terms vanish. Integrating on $(0,T'-s)$ we get
 \begin{align}\label{eq: heatrep}
 \til w(U_{T'-s})-\til w(U_{0})=\int_0^{T'-s} E_i\til w(U_\tau)\, dW^i_\tau.
 \end{align}
 Taking expectations, and since the Ito integral of an adapted process is a martingale, we have
 \begin{align*}
 \mathbb E_{(x,T')}[f(X_{T'-s})]-w(x,T')=\mathbb E[\til w(U_{T'-s})-\til w(U_{0})]=0.
 \end{align*}
Here we used that $\til w(U_0)=w(x,T')$ and $\til w(U_{T'-s})=w(X_{T'-s},s)=f(X_{T'-s})$.
\end{proof}

A further corollary is that the Wiener measure can be characterized by the heat kernels.
 \begin{cor}\label{cor: kernel}
 Let $0\leq\tau_1<\tau_2<\ldots<\tau_k\leq T'$ be a partition and $A_1,\ldots,A_k\subset M^n$ Borel sets. Then it holds
 \begin{align*}
 \mathbb P& [X_{\tau_j}\in A_j, j=1,\ldots,k]\\
&\ \qquad = \int_{A_k}\ldots\int_{A_1}p_{T',T'-\tau_1}(x,y_1)\cdots p_{T'-\tau_{k-1},T'-\tau_k}(y_{k-1},y_k)\, dV_{g(T'-\tau_1)}(y_1)\cdots \, dV_{g(T'-\tau_k)}(y_k).
 \end{align*}
\end{cor}

\subsection{Feynman-Kac formula}

\begin{prop}\label{prop: feynmankac}
Let $s\in[0,T']$, $A_t\in\mathrm{End}(TM)$ and $Y$ a vector valued solution of the heat equation with potential, $\nabla_tY=\Delta_{g_t}Y+A_tY$, with $Y|_s=Z\in C_0^\infty(TM)$, then
\begin{align}
Y(x,T')=\E_{(x,T')}[R_{T'-s}S_{T'-s}Z(X_{T'-s})],
\end{align}
where $R_\tau=R_\tau(\gamma)\colon T_xM\to T_xM$ is the solution of the ODE $\frac{d}{d\tau}R_\tau=R_\tau S_\tau A_{T'-\tau}S_\tau^{-1}$ with $R_0=\mathrm{id}$.
\end{prop}

\begin{proof}
Let $\til Y\colon \FF\to\R^n$, $\til Y(U)=u^{-1}Y_{\pi u}$. Applying the Ito formula \eqref{eq: Ito}, we obtain
\begin{align*}
d\til Y(U_\tau)=E_i\til Y(U_\tau)\, dW_\tau^i +\partial_\tau^*\til Y(U_\tau)\, d\tau+E_iE_i\til Y(U_\tau)\, d\tau
=E_i\til Y(U_\tau)\, dW_\tau^i-A_{T'-\tau}\til Y(U_\tau)\, d\tau,
\end{align*}
where we used Lemma \ref{lma: commute} and Proposition \ref{prop: commute}. Let $\til R_\tau\colon \R^n\to\R^n$ be the solution of the ODE $\frac{d}{d\tau}\til R_\tau \til A_{T'-\tau}$ with $R_0=\mathrm{id}$. Then
\begin{align}\label{eq: heatrep3}
d(\til R_\tau\til Y(U_\tau))=\til R_\tau H_i\til Y(U_\tau)\, dW_\tau^i.
\end{align}
Integrating on $[0,T'-s]$ and taking expectations, we obtain
\begin{align*}
\til Y(u)=\E_u[\til R_{T'-s}\til Y(U_{T'-s})].
\end{align*}
Finally, we compute
\begin{align*}
Y(x,T')=u\til Y(u)=\E_u[U_0\til R_{T'-s}U_0^{-1}U_0U_{T'-s}^{-1}U_{T'-s}\til Y(U_{T'-s})]=\E_u[\til R_{T'-s}\til S_{T'-s}Z(X_{T'-s})],
\end{align*}
since $U_{T'-s}\til Y(U_{T'-s})=Y(X_{T'-s},s)=Z(X_{T'-s})$ and $R_\tau=U_0\til R_\tau U_0^{-1}$. Indeed, the last equality holds since
\begin{align*}
\frac{d}{d\tau}(U_0\til R_\tau U_0^{-1})=U_0\til R_\tau\til A_{T'-\tau}=U_0\til R_\tau U_0^{-1}U_0U_\tau^{-1}\til A_{T'-\tau}U_\tau^{-1}U_\tau U_0^{-1}=U_0\til R_\tau U_0^{-1}S_\tau A_{T'-\tau}S_\tau^{-1},
\end{align*}
which shows that $R_\tau$ and $U_0\til R_\tau U_0^{-1}$ solve the same ODE, and thus must be equal.
\end{proof}

\subsection{Induced martingales and parallel gradients}

\begin{defn}
Let $F\colon P_{T'}\mathcal M\to\R$ be integrable. Then, we define the induced martingale as
\begin{align*}
 F_\tau(\gamma):=\mathbb E_{(x,T')}[F|\Sigma_\tau](\gamma).
\end{align*}
Note that then $F_\tau$ satisfies the martingale property, i.e. for all $\sigma\leq \tau$
\begin{align*}
\mathbb E_{(x,T')}[F_\tau|\Sigma_\sigma]=F_\sigma,
\end{align*}
by the definition of conditional expectation and that $\Sigma_\sigma\subset\Sigma_\tau$. 
\end{defn}

The next results concern the induced martingale of an integrable function $F$.  Note that integrability is not a big restriction, since each uniformly integrable martingale can be represented as the induced martingale of an integrable function.  Explicitly, by standard results the induced martingale satisfies the following:

\begin{prop}\label{prop: induced martingale}
	Let $F\colon P_{T'}\mathcal M\to \R$ be integrable. Then, for almost every Brownian curve $\{\gamma_\tau\}_{\tau\in[0,T']}$ we have for the induced martingale
	\begin{align*}
	F_\tau(\gamma):=\mathbb E_{(x,T')}[F|\Sigma_\tau](\gamma)=\int_{P_{\gamma_\tau\mathcal M}}F(\gamma|_{[0,\tau]}\ast \gamma')\, d\Gamma_{\gamma_\tau}(\gamma'),
	\end{align*}
	where we integrate over all $\gamma'$ in the based path space $P_{\gamma_\tau}\mathcal M$ and $\ast$ denotes the concatenation of the two curves $\gamma|_{[0,T']}$ and $\gamma'$.
\end{prop}

The analysis to follows exploits a nice set of nice test functions on path space, namely cylinder functions:

\begin{defn} 
Given $\pmb{\tau}=\{\tau_j\}_{j=1}^k$ be a partition of $[0,T']$ we define evaluation maps
\begin{align*}
e_{\pmb\tau}\colon P_{T'}\mathcal M\to M^k, \qquad e_{\pmb\tau}(\gamma)=(\pi_1\gamma_{\tau_1},\pi_1\gamma_{\tau_2},\ldots,\pi_1\gamma_{\tau_k}).
\end{align*}
Given a partition $\tau$ and a smooth compactly supported function $f\colon M^k\to\R$ we obtain a \emph{cylinder function}
\begin{align*}
F\colon P_{T'}\mathcal M\to\R, \qquad F(\gamma)=f(e_{\pmb\tau}(\gamma)).
\end{align*}
The space of all cylinder functions is denoted $\cyl$.
\end{defn}

\begin{defn} Let $F \in \cyl$ and fix $\gamma \in P_{T'} \mathcal M$.  Given $V$ a vector field along $\gamma$ we let $\xi^\eps=(x^\eps_\tau,T'-\tau)_{\tau\in[0,T']}$ denote a one-parameter family of curves such that $\xi^0 = \gamma$ and $\left. \frac{\del}{\del \eps} \right|_{\eps=0}x^\eps_\tau = V_\tau$.  Then
\begin{align*}
D_V F := \left. \frac{\del}{\del \eps} \right|_{\eps=0} f(e_{\pmb\tau}(\xi^\eps)).
\end{align*}
In our setting we will only use a special class of vector fields $V$.  In particular, let $\HH$ denote the Hilbert space of $H^1$-curves $(h_\tau)_{\tau\geq0}$ in $(T_{x} M,g_{T'})$ with $h_0=0$ equipped with the inner product
\begin{align*}
\IP{h_1,h_2}_\HH=\int_0^{T'} \IP{\dot h_1,\dot h_2}_{(T_xM,g_{T'})}\, d\tau.
\end{align*}  
Given $(h_\tau)_{\tau\geq0} \in \HH$ we let $V_{\tau}(\gamma) = S_{\tau}^{-1}(\gamma) h_{\tau}$.
\end{defn}

This derivative operator admits a key integration by parts formula, cf. \cite{Driver_ibp}, \cite[Theorem A.1]{HaslhoferNaber}.  In the statement below, for $(h_\tau)_{\tau\geq0}\in \HH$ we set
\begin{align*}
\IP{h_\tau,dW_\tau}=(U_0^{-1}h_\tau)_i\, dW_\tau^i,
\end{align*}
noting that this inner product is independent of the initial choice of frame $U_0$.  The theorem is proved in an appendix (\S \ref{s:IBP})
\begin{thm} \label{t:IBP}
Let $F,G\in\cyl$, let $(h_\tau)_{\tau\geq 0}\in\HH$ and write $V=(S_\tau^{-1}h_\tau)_{\tau\geq 0}$. Then
\begin{align}\label{eq: intbyparts}
D_V^*G:=-D_VG+\frac12 G\int_0^{T'}\IP{\frac{d}{d\tau}h_\tau-S_\tau(\Rc^\nabla+\frac12\partial_t(g-b))_{T-\tau}^\dagger S_\tau^{-1}h_\tau,\, dW_\tau}
\end{align}
satisfies $\E_{(x,T')}[D_VF G]=\E_{(x,T')}[F D_V^* G]$.
\end{thm}

\begin{defn}
Let $\sigma\in[0,T']$ and let $F\in\cyl$.
The \emph{parallel gradient} $\nabla_\sigma^\perp F\colon P_{(x,T')}\MM\to (T_xM,g_{T'})$ is defined as
\begin{align*}
D_{V^\sigma}F(\gamma)=\IP{\nabla_\sigma^\perp F(\gamma),v}_{(T_xM,g_{T'})},
\end{align*}
where $V_\tau^\sigma=S_\tau^{-1}v\mathds 1_{[\sigma,T']}(\tau)$ and $v\in(T_xM,g_{T'})$.
Explicitly, if we have the representation $F=f\circ e_{\pmb\tau}$, $\pmb\tau=\{\tau_j\}_{j=1}^k$, it follows that
\begin{align}\label{eq: parallel}
\nabla_\sigma^\perp F(\gamma)=\sum_{\tau_j\geq\sigma}S_{\tau_j}\grad^{(j)}_{g_{T'-\tau_j}}f(\pi_1\gamma_{\tau_1},\ldots,\pi_1\gamma_{\tau_k}).
\end{align}  
\end{defn}

\begin{defn}  Given $F\in\cyl$, its \emph{Malliavin derivative} $\nabla^\HH F\colon P_{(x,T')}\mathcal M\to \HH$ is defined as
\begin{align*}
D_VF(\gamma)=\IP{\nabla^\HH F, h}_\HH,
\end{align*}
for every $h\in \HH$ and $V=(S_\tau^{-1} h_\tau)_{\tau\geq 0}$. It follows that the parallel gradient is the time derivative of the Malliavin gradient $\frac{d}{d\tau}(\nabla^\HH F)_\tau=\nabla_\tau^\perp F$ and furthermore
\begin{align*}
|\nabla^\HH F|_\HH^2=\int_0^{T'}|\nabla_\tau^\perp F|^2\, d\tau.
\end{align*}
\end{defn}

\begin{defn} Given the setup above and $0 \leq \tau_1 \leq \tau_2 \leq T'$, we define the \emph{Ornstein-Uhlenbeck operator} 
\begin{align*}
\LL_{(\tau_1,\tau_2)} := \int_{\tau_1}^{\tau_2}{\nabla_\tau^\perp}^*\nabla_\tau^\perp\, d\tau.
\end{align*}
\end{defn}

\begin{rmk} Our discussion above and proofs below work exclusively with cylinder functions.  Due to the integration by parts formula \eqref{eq: intbyparts} the Malliavin gradient is closable and can be extended to a closed unbounded operator from $L^2(P_{(x,T')}\MM)$ to $L^2(P_{(x,T')}\MM,\HH)$ with $\cyl$ being a dense subset of the domain (cf. \cite{Hsu} Section 8).  The definitions of all derivative operators considered here can be similarly extended.
\end{rmk}

\subsection{Martingale representation theorem}

\begin{prop}\label{p:martingale}
	Let $F\in \cyl$ %D^{1,2}$
	 and let $F_\tau$ be the induced martingale. Then $F_\tau$ solves
	\begin{align*}
	dF_\tau=\IP{\nabla_\tau^\perp F_\tau,dW_\tau}, \qquad F|_{\tau=0}=F_0.
	\end{align*}
\end{prop}
\begin{rmk}
	This result shows that martingales are the natural generalization of the (backward) heat-flow to the path space $P\mathcal M$. Indeed, let $F_\tau(\gamma)=f_\tau(\pi_1\gamma_\tau)$ for some smooth function $f\colon\mathcal M\to\R$. Then Proposition \ref{p:martingale} together with  \eqref{eq: Ito} yield
	\begin{align*}
	0=df_\tau(\pi_1\gamma_\tau)-\IP{\nabla f_\tau(\pi_1\gamma_\tau),dW_\tau}=(\partial_\tau+\Delta_{g_{T-\tau}}) f_\tau(\pi_1\gamma_\tau),
	\end{align*}
	which means that $f_\tau$ solves the backward heat equation.
\end{rmk}
\begin{proof}
	Let $F(\gamma)=f(\pi_1\gamma_{\tau_1},\ldots,\pi_1\gamma_{\tau_k})$, where $f\colon M^k\to\R$ is a smooth compactly supported function. Then, for $\tau\in(\tau_l,\tau_{l+1})$, by Corollary \ref{cor: kernel} and Proposition \ref{prop: induced martingale} we have
	\begin{align*}
	F_\tau(\gamma)
	=&\ \int_{P_{\gamma_\tau}\mathcal M}F(\gamma|_{[0,\tau]}\ast \gamma')\, d\Gamma_{\gamma_\tau}(\gamma')\\
	=&\ \int_{P_{\gamma_\tau}\mathcal M}f(\pi_1\gamma_{\tau_1},\ldots,\pi_1\gamma_{\tau_l}, \pi_1\gamma_{\tau_{l+1}-\tau},\ldots, \pi_1\gamma_{\tau_k-\tau})\, d\Gamma_{\gamma_\tau}(\gamma')\\
	=&\ \int_{M^{k-l}}f(X_{\tau_1},\ldots,X_{\tau_l},y_{l+1},\ldots,y_k)p_{T'-\tau,T'-\tau_{l+1}}(X_\tau,y_{l+1})\ldots p_{T'-\tau_{k-1},T'-\tau_k}(y_{k-1},y_k)\\
	&\ \qquad dV_{g_{T'-\tau_{l+1}}}(y_{l+1})\ldots\, dV_{g_{T'-\tau_k}}(y_k)\\
	=:&\ f_\tau(X_{\tau_1},\ldots, X_{\tau_l},X_\tau).
	\end{align*}
	Note that for $(x_1,\ldots,x_l)$ fixed, $(x,\tau)\mapsto f_\tau(x_1,\ldots,x_l,x)$ solves $(\partial_\tau+\Delta^{(l+1)})f_\tau=0$, where $\Delta^{(l+1)}$ acts on the last entry.
	
	Let $\til f_\tau=f_\tau\circ \otimes_1^{l+1}\pi_1\circ\otimes_1^{l+1}\pi$ and $\til F_\tau=F_\tau\circ\Pi$. Then $\til F_\tau(U)=\til f_\tau(U_{\tau_1},\ldots,U_{\tau_l},U_\tau)$. According to \eqref{eq: Ito} we have then
	\begin{align*}
	d\til F_\tau(U)=&\ d\til f_\tau(U_{\tau_1},\ldots,U_{\tau_l},U_\tau)\\
	=&\ (\partial^*_\tau\til f_\tau(U_{\tau_1},\ldots,U_{\tau_l},U_\tau)+E^{(l+1)}_iE^{(l+1)}_i\til f_\tau(U_{\tau_1},\ldots,U_{\tau_l},U_\tau)\, d\tau\\
	&\ \qquad +\IP{E^{(l+1)}_i\til f_\tau(U_{\tau_1},\ldots,U_{\tau_l},U_\tau),\, dW_\tau^i}.
	\end{align*}
	Note that due to Proposition \ref{prop: commute} we have $\partial^*_\tau\til f_\tau+E^{(l+1)}_iE^{(l+1)}_i\til f_\tau=0$. Next, we compute
	\begin{align*}
	E_i^{(l+1)}\til f_\tau(U_{\tau_1},\ldots,U_{\tau_l},U_\tau)=&\ (U_\tau e_i)^*\til f_\tau(U_{\tau_1},\ldots,U_{\tau_l},U_\tau)\\
	=&\ (U_\tau e_i)f_\tau(X_{\tau_1},\ldots,X_{\tau_l},X_\tau)\\
	=&\ \IP{U_\tau e_i,\grad^{(l+1)}_{g_{T'-\tau}}f_\tau(X_{\tau_1},\ldots,X_{\tau_l},X_\tau)}_{(T_{X_\tau}M,g_{T'-\tau})}\\
	=&\ \IP{S_\tau U_\tau e_i,S_\tau\grad^{(l+1)}_{g_{T'-\tau}}f_\tau(X_{\tau_1},\ldots,X_{\tau_l},X_\tau)}_{(T_xM,g_{T'})}\\
	=&\ \IP{U_0 e_i,\nabla_{\tau}^\perp F_\tau(\gamma)}_{(T_{x}M,g_{T'})},
	\end{align*}
	where we used Lemma \ref{lma: commute} in the second line and \eqref{eq: parallel} in the last line. All in all we find
	\begin{align*}
	dF_\tau(\gamma)=d\til F_\tau(U)=\IP{U_0 e_i,\nabla_{\tau}^\perp F_\tau(\gamma)}_{(T_{x}M,g_{T'})}\, dW_\tau^i=\IP{\nabla_{\tau}^\perp F_\tau(\gamma),dW_\tau},
	\end{align*}
	which was the claim.
\end{proof}

\begin{cor}\label{cor: quadratic}
	Let $F\in\cyl$. Then the quadratic variation $[F,F]_\tau$ of the induced martingale $F_\tau$ satisfies
	\begin{align*}
	d[F,F]_\tau=2|\nabla_\tau F_\tau|^2\, d\tau.
	\end{align*}
\end{cor}
\subsection{Evolution of the parallel gradient}

In the next result we give is about the evolution of the parallel gradient. 

\begin{thm}\label{thm: bochner}
Let $F\in \cyl$ %D^{2,2}$
and let $\sigma\geq0$ be fixed. Then the parallel gradient of the induced martingale $\nabla_\sigma^\perp F_\tau$ satisfies
\begin{align*}
d(\nabla_\sigma^\perp F_\tau)=\IP{\nabla_\tau^\perp\nabla_\sigma^\perp F_\tau,\, dW_\tau}
+({\Rc^\nabla}+\frac12{\partial_t (g-b)})_\tau(\nabla_\tau^\perp F_\tau)\mathds{1}_{[\sigma,T']}(\tau)\, d\tau+\nabla_\sigma^\perp F_\sigma\delta_\sigma(\tau)\, d\tau,
\end{align*}
where $\IP{{\Rc^\nabla}+\frac12{\partial_t (g-b)})_\tau(v),w}_{T_xM,g_{T'}}=({\Rc_{g_t}^\nabla}+\frac12{\partial_t (g-b)})|_{t=T'-\tau}(S_\tau^{-1}v,S_\tau^{-1}w)$.
\end{thm}

\begin{proof}
Since $F_\tau$ is $\Sigma_\tau$-measurable, i.e. it depends only on times smaller than $\tau$, we have that $\nabla_\sigma^\perp F_\tau=0$ as soon as $\tau<\sigma$. At $\sigma=\tau$ we have a jump discontinuity, which is expressed in the $\delta$-notation $\nabla_\sigma^\perp F_\sigma\delta_\sigma(\tau)$. For $\tau>\sigma$ we aim to show the evolution 
\begin{align}\label{eq: Bochner}
d(\nabla_\sigma^\perp F_\tau)=\IP{\nabla_\tau^\perp\nabla_\sigma^\perp F_\tau,\, dW_\tau}
+({\Rc^\nabla}+\frac12{\partial_t (g-b)})_\tau(\nabla_\tau^\perp F_\tau)\, d\tau.
\end{align}
Let $F(\gamma)=f(\pi_1\gamma_{\tau_1},\ldots\pi_1\gamma_{\tau_k})$ be a cylinder function.
Let $\tau\in(\tau_l,\tau_{l+1})$, then $F_\tau(\gamma)=f_\tau(X_{\tau_1},\ldots,X_{\tau_k},X_{\tau})$ as in the proof of Proposition \ref{p:martingale} and by virtue of \eqref{eq: parallel}
\begin{align*}
\nabla_\sigma^\perp F_\tau(\gamma)=\sum_{\tau_j\geq\sigma}S_{\tau_j}\grad_{g_{T'-{\tau_j}}}^{(j)}f_\tau(X_{\tau_1},\ldots,X_{\tau_k},,X_\tau) + S_\tau \grad^{(l+1)}_{g_{T'-\tau}} f_\tau(X_{\tau_1},\ldots,X_{\tau_k},X_\tau).
\end{align*}
Consider $G_i(U):=\IP{U_0e_i,\nabla_\sigma^\perp F_\tau(\Pi U)}$. Then
\begin{align*}
G_i(U):=&\ \IP{U_0e_i,\nabla_\sigma^\perp F_\tau(\Pi U)}_{T_xM,g_{T'}}\\
=&\ \sum_{\tau_j\geq\sigma}\IP{U_{\tau_j} e_i,\grad^{(j)}_{g_{T'-\tau_j}} f_\tau(X_{\tau_1},\ldots,X_{\tau_k},X_\tau)}_{T_{X_{\tau_j}}M,g_{T'-\tau_j}}\\
&+\ \IP{U_\tau e_i,\grad^{(l+1)}_{g_{T'-\tau}} f_\tau(X_{\tau_1},\ldots,X_{\tau_k},X_\tau)}_{T_{X_\tau}M,g_{T'-\tau}}\\
=&\ \sum_{\tau_j\geq\sigma}E_i^{(j)}\til f_\tau(U_{\tau_1},\ldots,U_{\tau_l},U_\tau)
+
E^{(l+1)}_i\til f_\tau(U_{\tau_1},\ldots,U_{\tau_l},U_\tau),
\end{align*}
where we used Lemma \ref{lma: commute} in the third line. Then with \eqref{eq: Ito} we find
\begin{align*}
dG_i(U)=&\sum_{\tau_j\geq\sigma}\left(\partial_\tau^*E^{(j)}_i\til f_\tau(U_\tau)+E_m^{(l+1)}E_m^{(l+1)}E_i^{(j)}\til f_\tau(U_{\tau_1},\ldots,U_{\tau_l},U_\tau)\right)\, d\tau
+E^{(l+1)}_mE^{(j)}_i\til f_\tau(U_\tau)\, dW_\tau^m\\
&\ +\left(\partial_\tau^*E^{(l+1)}_i\til f_\tau(U_\tau)+E_m^{(l+1)}E_m^{(l+1)}E_i^{(l+1)}\til f_\tau(U_{\tau_1},\ldots,U_{\tau_l},U_\tau)\right)\, d\tau
+E^{(l+1)}_mE^{(l+1)}_i\til f_\tau(U_\tau)\, dW_\tau^m\\
=&\ \sum_{\tau_j\geq\sigma}E_i^{(j)}(\partial_\tau^*+E^{(l+1)}_mE^{(l+1)}_m)\til f_\tau(U_\tau)\, d\tau
+E^{(l+1)}_mE^{(j)}_i\til f_\tau(U_\tau)\, dW_\tau^m\\
&\ +E^{(l+1)}_i(\partial_\tau^*+E^{(l+1)}_mE^{(l+1)}_m)\til f_\tau(U_\tau)\, d\tau
+E^{(l+1)}_mE^{(l+1)}_i\til f_\tau(U_\tau)\, dW_\tau^m
+[\partial_\tau^*+E^{(l+1)}_mE^{(l+1)}_m,E_i]\til f_\tau(U_\tau)\, d\tau.
\end{align*}
Recall that $(\partial_\tau^*+E^{(l+1)}_mE^{(l+1)}_m)\til f_\tau=0$ due to Proposition \ref{prop: commute}.  Furthermore, using Propositions \ref{p:gradientev} and \ref{prop: commute} we deduce
\begin{align*}
[\partial_\tau^*+E^{(l+1)}_mE^{(l+1)}_m,E_i]\til f_\tau=(\widetilde{\Rc^\nabla}+\frac12\widetilde{\partial_t (g-b)})_{im}E^{(l+1)}_m\til f_\tau.
\end{align*}
All in all this gives us
\begin{align*}
dG_i(U)
=&\sum_{\tau_j\geq\sigma}E^{(l+1)}_mE^{(j)}_i\til f_\tau(U_\tau)\, dW_\tau^m
+E^{(l+1)}_mE^{(l+1)}_i\til f_\tau(U_\tau)\, dW_\tau^m
+(\widetilde{\Rc^\nabla}+\frac12\widetilde{\partial_t (g-b)})_{im}E^{(l+1)}_m\til f_\tau
\end{align*}
Projecting down yields
\begin{align*}
&\sum_{\tau_j\geq\sigma}E^{(l+1)}_mE^{(j)}_i\til f_\tau(U_\tau)\, dW_\tau^m
+E^{(l+1)}_mE^{(l+1)}_i\til f_\tau(U_\tau)\, dW_\tau^m\\
&\ \quad = \IP{\sum_{\tau_j\geq\sigma}(S_\tau\otimes S_{\tau_j})\nabla^{(l+1)}\nabla^{(j)} f_\tau
+(S_\tau\otimes S_\tau)\nabla^{(l+1)}\nabla^{(l+1)}f_\tau,dW_\tau\otimes U_0e_i}\\
&\ \quad = \IP{\nabla_\tau^\perp\nabla_\sigma^\perp F_\tau(\gamma),dW_\tau\otimes U_0e_i},
\end{align*}
and 
\begin{align*}
(\widetilde{\Rc^\nabla}+\frac12\widetilde{\partial_t(g-b)})_{im}E_m^{(l+1)}\til f_\tau(U_\tau)\, d\tau
=\IP{(\Rc^\nabla+\frac12\partial_t(g-b))_\tau(\nabla_\tau^\perp F_\tau)\, d\tau, U_0e_i},
\end{align*}
giving the result.
\end{proof}

\begin{lemma}\label{lemma: hessian}
Let $F\in\cyl$ and $\tau,\sigma\geq0$ fixed. Then
\begin{align*}
\nabla_\tau^\perp |\nabla_\sigma^\perp F|^2=2\IP{\nabla_\tau^\perp \nabla_\sigma^\perp F,\nabla_\sigma^\perp F}.
\end{align*}	
\end{lemma}	
\begin{proof}
	Let $F(\gamma)=f(\pi_1\gamma_{\tau_1},\cdots,\pi_1\gamma_{\tau_k})$. Then, as in the proof of Theorem \ref{thm: bochner}
	\begin{align*}
	\IP{\nabla_\sigma^\perp F(\gamma),U_0e_a}=\sum_{\tau_j\geq\sigma}E_a^{(j)}\til f(U_{\tau_1},\ldots,U_{\tau_k}).
	\end{align*}
	Hence
	\begin{align*}
	\IP{\nabla_\tau^\perp|\nabla_\sigma^\perp F|^2,U_0e_b}=&\sum_{\tau_k\geq \tau}E_b^{(k)}\sum_{a=1}^n(\sum_{\tau_j\geq\sigma}E_a^{(j)}\til f)^2\\
	=&2\sum_{a=1}^n\sum_{\tau_k\geq\tau}\sum_{\tau_j\geq\sigma}E_b^{(k)}E_a^{(j)}\til f(\sum_{\tau_j\geq\sigma}E_a^{(j)}\til f).
	\end{align*}
	Projecting down and using Proposition \ref{prop: commute} yields the claim.
\end{proof}

\begin{cor}\label{c: bochner2}
Let $F\in \cyl$
and $\sigma\geq0$ fixed. Then $\nabla_\sigma^\perp F_\tau\colon P_{(x,T')}\mathcal M\to(T_xM,g_{T'})$ satisfies
\begin{enumerate}
	\item the quadratic Bochner identity
	\begin{align*}
	d(|\nabla_\sigma^\perp F_\tau|^2)=&\IP{\nabla_\tau^\perp|\nabla_\sigma^\perp F_\tau|^2, dW_\tau}
	+2(\Rc^\nabla+\frac12\partial(g-b))_\tau(\nabla_\tau^\perp F_\tau,\nabla_\sigma^\perp F_\tau)\, d\tau\\
	&+2|\nabla_\tau^\perp\nabla_\sigma^\perp F_\tau|^2\, d\tau+|\nabla_\tau^\perp F_\tau|^2\delta_\sigma(\tau)\, d\tau,
	\end{align*}
	\item and the linear Bochner identity
	\begin{align*}
	d|\nabla_\sigma^\perp F_\tau|=&\IP{\nabla_\tau^\perp|\nabla_\sigma^\perp F_\tau|,dW_\tau}+\frac{|\nabla_\tau^\perp\nabla_\sigma^\perp F_\tau|^2-|\nabla_\tau^\perp|\nabla_\sigma^\perp F_\tau||^2}{|\nabla_\sigma^\perp F_\tau|}\, d\tau\\
	&+\frac1{|\nabla_\sigma^\perp F_\tau|}(\Rc^\nabla+\frac12\partial_t(g-b))_\tau(\nabla_\tau^\perp F_\tau,\nabla_\sigma^\perp F_\tau)\, d\tau+|\nabla_\tau^\perp F_\tau|\delta_\sigma(\tau)\, d\tau.
	\end{align*}
	
\end{enumerate}

Here, we denote $(\Rc^\nabla+\frac12\partial(g-b))_\tau(v,w)=(\Rc^\nabla_{g_t}+\frac12\partial_t(g-b))|_{t=T'-\tau}(S_\tau^{-1}v,S_\tau^{-1} w)$.
\end{cor}

\begin{proof}
As in the previous proof, it is enough to consider the case $\sigma<\tau$. By Ito's Lemma and Theorem \ref{thm: bochner} we have 
\begin{align*}
d(|\nabla_\sigma^\perp F_\tau|^2)=&2\IP{\nabla_\sigma^\perp F_\tau,d(\nabla_\sigma^\perp F_\tau)}
+d[\nabla_\sigma^\perp F_\tau,\nabla_\sigma^\perp F_\tau]\\
=&2\IP{\nabla_\sigma^\perp F_\tau,\IP{\nabla_\tau^\perp\nabla_\sigma^\perp F_\tau,\, dW_\tau}+({\Rc^\nabla}+\frac12{\partial_t (g-b)})_\tau(\nabla_\tau^\perp F_\tau)\, d\tau}\\
&+2|\nabla_\tau^\perp\nabla_\sigma^\perp F_\tau|^2\, d\tau.
\end{align*}
Noticing that $2\IP{\nabla_\sigma^\perp F_\tau,\IP{\nabla_\tau^\perp\nabla_\sigma^\perp F_\tau,\, dW_\tau}}=\IP{\nabla_\tau^\perp|\nabla_\sigma^\perp F_\tau|^2, dW_\tau}$
due to Lemma \ref{lemma: hessian}, this proves the quadratic Bochner identity. 

In order to show the linear Bochner identity, we use the Ito-Tanaka-Meyer formula for the convex function $|\cdot|\colon\R^n\to \R$, cf. \cite{haslhofer2018ricci}. Let us note that there is no local time at the origin, since we assume dimension $>1$.
\end{proof}

\begin{cor} \label{c:Bochnerreduction} The generalized Bochner formula on $P\mathcal M$ reduces to
\begin{align*}
\frac12(\partial_\tau + \Delta_{g_{T'-\tau}})|\nabla f_\tau|^2
=|\nabla\nabla f_\tau|^2+(\Rc^\nabla+\frac12\partial_t (g-b))|_{t=T'-\tau}(\nabla f_\tau,\nabla f_\tau),
\end{align*}
where $f_\tau=P_{T'-\tau,T'-\tau_1}f$, $f\colon M\to\R$ is a smooth function, $0<\tau_1<T'$ is fixed, and $\tau <\tau_1$.

\end{cor}
\begin{proof}
Define 
\begin{align*}
F_\tau(\gamma)=\begin{cases}
P_{T'-\tau,T'-\tau_1}f(\pi_1\gamma_\tau) & \text{ if } \tau<\tau_1\\
f(\pi_1\gamma_{\tau_1}) & \text{ if } \tau\geq \tau_1.
\end{cases}
\end{align*}
It follows from Proposition \ref{prop: induced martingale} that this defines a martingale on $P\mathcal M$. Moreover, 
\begin{align*}
|\nabla_0^\perp F_\tau|(\gamma)=|\nabla_\tau^\perp F_\tau|(\gamma)=|\nabla f_\tau|(\pi_1\gamma_\tau)
\end{align*}
and
\begin{align*}
|\nabla_0^\perp\nabla_\tau^\perp F_\tau|(\gamma)=|\nabla\nabla f_\tau|(\pi_1\gamma_\tau)
\end{align*}
By virtue of Corollary \ref{c: bochner2} we have
\begin{align*}
d(|\nabla f_\tau|^2)-\IP{\nabla_\tau^\perp|\nabla f_\tau|^2,dW_\tau}=2|\nabla\nabla f_\tau|^2\, d\tau+2(\Rc^\nabla+\frac12\partial_t(g-b))|_{t=T'-\tau}(\nabla f_\tau, \nabla f_\tau)\, d\tau,
\end{align*}
with all quantities evaluated at $\pi_1\gamma_\tau$. Then by Ito's formula \eqref{eq: Ito}, eventually lifting everything on the frame bundle, we find for the left hand side
\begin{align*}
d(|\nabla f_\tau|^2)-\IP{\nabla_\tau^\perp|\nabla f_\tau|^2,dW_\tau}=(\partial_\tau+\Delta_{g_{T'-\tau}})|\nabla f_\tau|^2\, d\tau.
\end{align*}
All in all we obtain
\begin{align*}
(\partial_\tau+\Delta_{g_{T'-\tau}})|\nabla f_\tau|^2\, d\tau=2|\nabla\nabla f_\tau|^2\, d\tau+2(\Rc^\nabla+\frac12\partial_t(g-b))|_{t=T'-\tau}(\nabla f_\tau, \nabla f_\tau)\, d\tau,
\end{align*}
which holds at $\pi_1\gamma_\tau$ for $\Gamma_x$-a.e. curves $\gamma$, which means by the defintition of the Wiener measure $\Gamma_x$ it holds for a.e. $y\in M$. By smoothness of $f$ we obtain the claim for every $y\in M$.
\end{proof}

\begin{lemma}\label{lemma: hessian2}
Let $F\in\cyl$ be nonnegative and $\tau,\sigma\geq0$ fixed. Then
\begin{align*}
\nabla_\sigma^\perp \nabla_\tau^\perp\log F=F^{-1}\nabla_\sigma^\perp\nabla_\tau^\perp F-F^{-2}\nabla_\tau^\perp F\otimes \nabla_\sigma^\perp F.
\end{align*}
\end{lemma}
\begin{proof}
The proof follows by computing on the frame bundle as in the proof of Lemma \ref{lemma: hessian}.
\end{proof}

\begin{cor}\label{c: bochner3}
	Let $F\in\cyl$ be nonnegative and let $F_\tau$ be the induced martingale. Then $X_\tau:=F_\tau^{-1}|\nabla^\HH F_\tau|^2-F_\tau\log F_\tau$ satisfies
	\begin{align*}
	dX_\tau=&\IP{\nabla_\tau^\perp X_\tau,dW_\tau}+2F_\tau\left(\int_0^{T'}|\nabla_\tau^\perp\nabla_\sigma^\perp \log F_\tau|^2\, d\sigma\right)\, d\tau\\
	&+2F_{\tau}^{-1}\left(\int_0^{T'}(\Rc^\nabla +\frac12\partial_t(g-b))_\tau(\nabla_\sigma^\perp F_\tau,\nabla_\tau^\perp F_\tau)\, d\sigma\right)\, d\tau.
	\end{align*}
\end{cor}
\begin{proof}
	Note that 
	\begin{align}\label{eq: x1}
	d(F_\tau\log F_\tau)=\IP{\nabla_\tau^\perp(F_\tau\log F_\tau),dW_\tau}
	+F_\tau^{-1}|\nabla_\tau^\perp F_\tau|^2\, d\tau
    \end{align}
    due to Proposition \ref{p:martingale} and the standard Ito formula. For the other term we compute
    \begin{align*}
    d(F_\tau^{-1}|\nabla^\HH F_\tau|^2)=F_\tau^{-1}d|\nabla^\HH F_\tau|^2
    +|\nabla^\HH F_\tau|^2dF_\tau^{-1}+d[F_\tau^{-1},|\nabla^\HH F_\tau|^2].
    \end{align*}
    Noticing that
    \begin{align*}
    d|\nabla^\HH F_\tau|^2=&\IP{\nabla_\tau^\perp |\nabla^\HH F_\tau|^2,dW_\tau}\\
    &+2\int_0^{T'}\left((\Rc^\nabla+\frac12\partial_t(g-b))_\tau(\nabla_\tau F_\tau,\nabla_\sigma^\perp F_\tau)+|\nabla_\tau^\perp\nabla_\sigma^\perp F_\tau|^2\right)\, d\sigma\, d\tau+|\nabla_\tau^\perp F_\tau|^2\, d\tau
    \end{align*}
    due to Corollary \ref{c: bochner2} and that
    \begin{align*}
    dF_\tau^{-1}=\IP{\nabla_\tau^\perp(F_\tau^{-1}),dW_\tau}+2F_\tau^{-3}|\nabla_\tau^\perp F_\tau|^2\, d\tau
    \end{align*}
    we compute using Lemma \ref{lemma: hessian} and Lemma \ref{lemma: hessian2}
    \begin{equation}\label{eq: x2}
    \begin{aligned}
    d(F_\tau^{-1}|\nabla^\HH F_\tau|^2)=&\IP{\nabla_\tau^\perp (F_\tau^{-1}|\nabla^\HH F_\tau|^2),dW_\tau}
    +2F_\tau\left(\int_0^{T'}|\nabla_\tau^\perp\nabla_\sigma\log F_\tau|^2\, d\sigma\right)\, d\tau\\
    &+2F_\tau^{-1}\left(\int_0^{T'}(\Rc^\nabla+\frac12\partial_t(g-b))_\tau(\nabla_\sigma F_\tau,\nabla_\tau F_\tau)\, d\sigma\right)\, d\tau\\
    &+F_\tau^{-1}|\nabla_\tau^\perp F_\tau|^2\, d\tau.
    \end{aligned}
\end{equation}
    All in all, combining \eqref{eq: x1} with \eqref{eq: x2} we get
    \begin{align*}
    dX_\tau=&\IP{\nabla_\tau^\perp X_\tau,dW_\tau}+2F_\tau\left(\int_0^{T'}|\nabla_\tau^\perp\nabla_\sigma^\perp \log F_\tau|^2\, d\sigma\right)\, d\tau\\
    &+2F_{\tau}^{-1}\left(\int_0^{T'}(\Rc^\nabla +\frac12\partial_t(g-b))_\tau(\nabla_\sigma^\perp F_\tau,\nabla_\tau^\perp F_\tau)\, d\sigma\right)\, d\tau.
    \end{align*}
\end{proof}

\section{Characterizations of generalized Ricci flow}

\begin{thm}\label{theorem: char}
For an evolving family of manifolds $(M,g_t,H_t)_{t\in[0,T]}$, the following are equivalent: 
\begin{enumerate}
\item \label{thmmain: 1} The generalized Ricci flow is satisfied
\begin{align*}
\partial_t(g-b)=-2\Rc^\nabla.
\end{align*}
\item \label{thmmain: 2}
Let $0\leq\sigma\leq T'\leq T$ and $F\in\cyl$. % D^{2,2}\cap D^{1,\infty}$
 Then the induced martingales satisfy the Bochner inequality
\begin{align*}%\label{eq: fullbochner}
d|\nabla_\sigma^\perp F_\tau|^2\geq \IP{\nabla_\tau^\perp|\nabla_\sigma^\perp F_\tau|^2,dW_\tau}
+2|\nabla_\tau^\perp\nabla_\sigma^\perp F_\tau|^2\, d\tau
+|\nabla_\sigma^\perp F_\sigma|^2\delta_\sigma(\tau)\, d\tau.
\end{align*}
\item \label{thmmain: 3}
Let $0\leq\sigma\leq T'\leq T$ and $F\in\cyl$. % D^{2,2}\cap D^{1,\infty}$ 
 Then the induced martingales satisfy the weak Bochner inequality
\begin{align*}%\label{eq: weakbochner}
d|\nabla_\sigma^\perp F_\tau|^2\geq \IP{\nabla_\tau^\perp|\nabla_\sigma^\perp F_\tau|^2,dW_\tau}
+|\nabla_\sigma^\perp F_\sigma|^2\delta_\sigma(\tau)\, d\tau.
\end{align*}
\item\label{thmmain: 4}
Let $0\leq\sigma\leq T'\leq T$ and $F\in \cyl$. %D^{2,2}\cap D^{1,\infty}$
 % such that $\tau\mapsto |\nabla_\sigma^\perp F_\tau|\in D^{1,2}$.
  Then the induced martingales satisfy the linear Bochner inequality
\begin{align*}
d|\nabla_\sigma^\perp F_\tau|\geq \IP{\nabla_\tau^\perp|\nabla_\sigma^\perp F_\tau|,dW_\tau}
+|\nabla_\sigma^\perp F_\sigma|\delta_\sigma(\tau)\, d\tau.
\end{align*}
\item\label{thmmain: 5}
Let $0\leq\sigma\leq T'\leq T$ and $F\in \cyl$. %D^{2,2}\cap D^{1,\infty}$ 
 %such that $\tau\mapsto |\nabla_\sigma^\perp F_\tau|\in D^{1,2}$. 
Then the induced martingales satisfy 
\begin{align*}
\tau\mapsto |\nabla_\sigma^\perp F_\tau| \text{ is a submartingale.}
\end{align*}
\end{enumerate}
\end{thm}

\begin{proof}
\eqref{thmmain: 1} $\Rightarrow$ \eqref{thmmain: 2}:
Let $\partial_t(g-b)=-2\Rc^\nabla$. Then the claim directly follows from Corollary \ref{c: bochner2}.

\eqref{thmmain: 2} $\Rightarrow$ \eqref{thmmain: 3}:
This follows from omitting the Hessian part.

\eqref{thmmain: 2} $\Rightarrow$ \eqref{thmmain: 4}:
Assume $\tau>\sigma$. Note that due to Corollary \ref{c: bochner2}, we have that
\begin{equation}\begin{aligned}\label{eq: linearbochner}
d|\nabla_\sigma^\perp F_\tau|=&\IP{\nabla_\tau^\perp|\nabla_\sigma^\perp F_\tau|,dW_\tau}+\frac{|\nabla_\tau^\perp\nabla_\sigma^\perp F_\tau|^2-|\nabla_\tau^\perp|\nabla_\sigma^\perp F_\tau||^2}{|\nabla_\sigma^\perp F_\tau|}\, d\tau\\
&+\frac1{|\nabla_\sigma^\perp F_\tau|}(\Rc^\nabla+\frac12\partial_t(g-b))_\tau(\nabla_\tau^\perp F_\tau,\nabla_\sigma^\perp F_\tau)\, d\tau.
\end{aligned}\end{equation}
 Now, comparing \eqref{eq: linearbochner} with \eqref{thmmain: 2} by standard Ito's lemma and using that $|\nabla_\tau^\perp|\nabla_\sigma^\perp F_\tau||^2\leq |\nabla_\tau^\perp\nabla_\sigma^\perp F_\tau|^2$ we directly find that
\begin{align*}
d|\nabla_\sigma^\perp F_\tau|\geq&\IP{\nabla_\tau^\perp|\nabla_\sigma^\perp F_\tau|,dW_\tau}.
\end{align*}
 Together with $\nabla_\sigma^\perp F_\tau=0$ for $\tau<\sigma$, this yields \eqref{thmmain: 4}.
 
 \eqref{thmmain: 4} $\Rightarrow$ \eqref{thmmain: 3}:
 This follows directly by appling Ito's formula and \eqref{thmmain: 4}.

\eqref{thmmain: 4} $\Leftrightarrow$ \eqref{thmmain: 5}:
Clearly, \eqref{thmmain: 4} implies \eqref{thmmain: 5}. For \eqref{thmmain: 5} $\Rightarrow$ \eqref{thmmain: 4}, note that \eqref{thmmain: 5} implies that the absolutely continuous part in \eqref{eq: linearbochner} must be nonnegative, which deduces \eqref{thmmain: 4}.

\eqref{thmmain: 3} $\Rightarrow$ \eqref{thmmain: 1}:
Let $F\colon P_{(x,T')}\mathcal M\to\R$. By Corollary \ref{c: bochner2} we know that
\begin{align*}
\tau\mapsto |\nabla_0^\perp F_\tau|^2-\int_0^\tau (\Rc^\nabla+\frac12\partial_t(g-b))_\rho(\nabla_\rho^\perp F_\rho,\nabla_0^\perp F_\rho)+2|\nabla_\rho^\perp\nabla_0^\perp F_\rho|^2\, d\rho
\end{align*}
is a martingale and consequently
\begin{align*}
\mathbb E_{(x,T')}\left[ |\nabla_0^\perp F_\varepsilon|^2-\int_0^\varepsilon (\Rc^\nabla+\frac12\partial(g-b))_\tau(\nabla_\tau^\perp F_\tau,\nabla_0^\perp F_\tau)+2|\nabla_\tau^\perp\nabla_0^\perp F_\tau|^2\, d\rho\right]=|\nabla_0^\perp F_0|^2.
\end{align*}
By virtue of \eqref{thmmain: 3} $\tau\mapsto |\nabla_0^\perp F_\tau|^2$ is a submartingale and in particular
\begin{align*}
\mathbb E_{(x,T')}[|\nabla_0^\perp F_\varepsilon |^2]\geq |\nabla_0^\perp F_0|^2.
\end{align*}
Together this implies that
\begin{align}\label{eq: weakbochner1}
\mathbb E_{(x,T')}\left[\int_0^\varepsilon (\Rc^\nabla+\frac12\partial(g-b))_\tau(\nabla_\tau^\perp F_\tau,\nabla_0^\perp F_\tau)+2|\nabla_\tau^\perp\nabla_0^\perp F_\tau|^2\, d\tau\right]\geq0.
\end{align}
 Now we consider two choices of cylinder function for $F$. Let the first one be $f_1\colon M\to\R$ with $f_1(x)=0$, $\nabla f_1(x)=v$, and $\nabla^2f_1(x)=0$, where $v\in (T_xM,g_{T'})$. Then define $F\colon P_{(x,T')}\mathcal M\to\R$ by $F(\gamma)=f_1(\pi_1(\gamma_\varepsilon))$.
For $\tau\leq\varepsilon$ we have
\begin{align*}
\nabla_0^\perp F_\tau=\nabla_\tau^\perp F_\tau=S_\tau\nabla P_{T'-\tau,T'-\varepsilon} f_1(\pi_1(\gamma_\tau)),\quad
|\nabla_\tau^\perp\nabla_0^\perp F_\tau|=|\nabla^2 P_{T'-\tau,T'-\varepsilon}f_1(\pi_1(\gamma_\tau))|
\end{align*}
and thus $\nabla_\tau^\perp F_\tau=v+o(\eps)$ and $|\nabla_\tau^\perp\nabla_0^\perp F_\tau|=o(\eps)$.

The second choice is $f_2\colon M\times M\to\R$ with $f_2(x,x)=0$, $\nabla^{(1)}f_2(x,x)=2v$, $\nabla^{(2)}f_2(x,x)=-v$, and $\nabla^2 f_2(x,x)=0$. Let $F(\gamma)=f_2(\pi_1(\gamma_0),\pi_1(\gamma(\eps)))$. Then, for $\tau\leq \eps$, 
\begin{align*}
F_\tau(\gamma)=&\ P_{T'-\tau,T'-\eps}^{(2)}f_2(x,\pi_1\gamma_\tau),\\
\nabla_0^\perp F_\tau=&\ \nabla^{(1)}P_{T'-\tau,T'-\eps}^{(2)}f_2(x,\pi_1\gamma_\tau)+S_\tau \nabla^{(2)}P_{T'-\tau,T'-\eps}^{(2)}f_2(x,\pi_1\gamma_\tau)=v+o(\eps),\\
\nabla_\tau^\perp F_\tau=&\ S_\tau \nabla^{(2)}P_{T'-\tau,T'-\eps}^{(2)}f_2(x,\pi_1\gamma_\tau)=-v+o(\eps),\\
|\nabla_\tau^\perp\nabla_0^\perp F_\tau| \leq&\ |\nabla^{(2)}\nabla^{(1)} P_{T'-\tau,T'-\eps}^{(2)}f_2(x,\pi_1\gamma_\tau)|
+|\nabla^{(2)} \nabla^{(2)}P_{T'-\tau,T'-\eps}^{(2)}f_2(x,\pi_1\gamma_\tau)|=o(\eps).
\end{align*}
Inserting both choices into \eqref{eq: weakbochner1} we get
\begin{align*}
 (\Rc^\nabla+\frac12\partial(g-b))_\eps(v,v)= (\Rc^\nabla+\frac12\partial_t(g-b))|_{t=T'-\eps}(S_\eps^{-1}v,S_\eps^{-1}v)=o(\eps).
\end{align*}
Letting $\eps\to0$ we get $ (\Rc^\nabla+\frac12\partial_t(g-b))|_{t=T'}=0$.
\end{proof}

\begin{thm}\label{thm: grad}
For an evolving family of manifolds $(M,g_t,H_t)_{t\in[0,T]}$, the following are equivalent:
\begin{enumerate}
\item\label{thmmain2: 1} The generalized Ricci flow 
\begin{align*}
\partial_t(g-b)=-2\Rc^\nabla
\end{align*}
is satisfied.
\item\label{thmmain2: 2}
Let $0\leq\sigma\leq T'\leq T$ and $F\in \cyl$. Then the induced martingales satisfy the gradient estimate
\begin{align*}
|\nabla_\sigma^\perp F_{\tau_1}|\leq \mathbb E_{(x,T')}[|\nabla_\sigma^\perp F_{\tau_2}||\Sigma_{\tau_1}]
\end{align*}
for all $0\leq\tau_1\leq \tau_2\leq T'$ and $x\in M$.
\item\label{thmmain2: 3}
Let $0\leq\sigma\leq T'\leq T$ and $F\in \cyl$. Then the induced martingales satisfy the gradient estimate
\begin{align*}
|\nabla_\sigma^\perp F_{\tau_1}|^2\leq \mathbb E_{(x,T')}[|\nabla_\sigma^\perp F_{\tau_2}|^2|\Sigma_{\tau_1}]
\end{align*}
for all $0\leq\tau_1\leq \tau_2\leq T'$ and $x\in M$.
	\item\label{thmmain2: 4}
		Let $0\leq\sigma\leq T'\leq T$ and $F\in \cyl$. Then we have the gradient estimate
		\begin{align*}
	|\nabla_x\mathbb E_{(x,T')}[ F]|^2\leq \mathbb E_{(x,T')}[|\nabla_0^\perp F|^2]
		\end{align*}
		for all $x\in M$. 
		\item\label{thmmain2: 5}
		For any $T'\in[0,T]$, $F\in \cyl$, the induced martingales satisfy the quadratic variation estimate
		\begin{align*}
	    \mathbb E_{(x,T')}\left[\frac{d[F,F]_\tau}{d\tau}\right]\leq 2\mathbb E_{(x,T')}[	|\nabla_\tau^\perp F|^2].
		\end{align*}
		for all $\tau\in[0,T']$ and $x\in M$.
\end{enumerate}
\end{thm}

\begin{proof}
\eqref{thmmain2: 1} $\Rightarrow$ \eqref{thmmain2: 2} $\Rightarrow$ \eqref{thmmain2: 3}:
\eqref{thmmain2: 2} immediately follows from \eqref{eq: linearbochner} and integrating in $\tau$ and taking expectations. Claim \eqref{thmmain2: 3} follows then from \eqref{thmmain2: 2} by Cauchy-Schwarz:
\begin{align*}
|\nabla_\sigma^\perp F_{\tau_1}|^2\leq \left(\mathbb E_{(x,T')}\left[|\nabla_\sigma^\perp F_{\tau_2}||\Sigma_{\tau_1}\right]\right)^2\leq \mathbb E_{(x,T')}\left[|\nabla_\sigma^\perp F_{\tau_2}|^2|\Sigma_{\tau_1}\right]
\end{align*}
for all $\tau_2\geq\tau_1$.

 \eqref{thmmain2: 3} $\Rightarrow$ \eqref{thmmain2: 5}:
Note that according to Theorem \ref{thm: bochner} $d[F,F]_\tau=2|\nabla_\tau^\perp F_\tau|^2\, d\tau$ and hence \eqref{thmmain2: 3} yields
\begin{align*}
\E_{(x,T')}\left[\frac{d[F,F]_\tau}{d\tau}\right]=2\E_{(x,T')}[|\nabla_\tau^\perp F_\tau|^2]
\leq 2\E_{(x,T')}\left[\E_{(x,T')}[|\nabla_\tau^\perp F|^2|\Sigma_\tau]\right]
= 2\E_{(x,T')}[|\nabla_\tau^\perp F|^2].
\end{align*}

\eqref{thmmain2: 5} $\Rightarrow$ \eqref{thmmain2: 4}:
This follows by recalling from Corollary \ref{cor: quadratic} that $\frac{d}{d\tau}[F,F]_\tau=2|\nabla_\tau^\perp F_\tau|^2$ and applying \eqref{thmmain2: 5}.

\eqref{thmmain2: 4} $\Rightarrow$ \eqref{thmmain2: 1}:
This follows by choosing 1-point and 2-point cylinder functions similarly as in the proof of the implication Theorem \ref{theorem: char} \eqref{thmmain: 3} $\Rightarrow$ \eqref{thmmain: 1}. 
\end{proof}

%%%

%%%
\begin{thm12}[] \label{thm: hessian}
	For an evolving family of manifolds $(M,g_t,H_t)_{t\in[0,T]}$, the following are equivalent:
	\begin{enumerate}
		\item\label{thmmain3: 1} The generalized Ricci flow 
		\begin{align*}
		\partial_t(g-b)=-2\Rc^\nabla
		\end{align*}
		is satisfied.
		\item\label{thmmain3: 2}
		For any $0\leq\sigma\leq T'\leq T$ and any $F\in \cyl$, we have the estimate
		\begin{align*}
	\mathbb E_{(x,T')}[	|\nabla_\sigma^\perp F_{\sigma}|^2]+2 \int_0^{T'} \mathbb E_{(x,T')}[|\nabla_\tau^\perp \nabla_\sigma^\perp F_\tau|^2]\, d\tau \leq \mathbb E_{(x,T')}[|\nabla_\sigma^\perp F|^2] 
		\end{align*}
		 for all $ x\in M$.
			\item\label{thmmain3: 6} For any $0\leq\tau_1\leq \tau_2\leq T'\leq T$ the Ornstein-Uhlenbeck operator $\LL_{(\tau_1,\tau_2)}$ on parabolic path space $L^2(P_{T'}\MM)$ satisfies the Poincar\'e inequality
		\begin{align*}
		\E_{(x,T')}[(F_{\tau_2}-F_{\tau_1})^2]\leq 2\E_{(x,T')}[F\, \LL_{(\tau_1,\tau_2)}F]
		\end{align*}
		for all $x\in M$.
		\item\label{thmmain3: 5} For any $0\leq\tau_1\leq \tau_2\leq T'\leq T$ the Ornstein-Uhlenbeck operator $\LL_{(\tau_1,\tau_2)}$ on parabolic path space $L^2(P_{T'}\MM)$ satisfies the log-Sobolev inequality
		\begin{align*}
		&\mathbb E_{(x,T')}[(F^2)_{\tau_2}\log((F^2)_{\tau_2})-(F^2)_{\tau_1}\log((F^2)_{\tau_1})]
		\leq4\mathbb E_{(x,T')}[F\, \LL_{(\tau_1,\tau_2)}F]
		\end{align*}
		for all $x\in M$.
	\end{enumerate}
Moreover, if one of the conditions \eqref{thmmain3: 1}-\eqref{thmmain3: 5} is satisfied, we have:
\begin{enumerate}
	\item[(3a)]{\label{thmmain3: 3a}}
	For any $0\leq T'\leq T$, $F\in \cyl$, we have the Poincar\'e Hessian estimate
	\begin{align*}
	\mathbb E_{(x,T')}[(F-\mathbb E_{(x,T')}[F])^2]+4\int_0^{T'}\int_0^{T'}\mathbb E_{(x,T')}[	|\nabla_\tau^\perp \nabla_\sigma^\perp F_{\tau}|^2]\, d\sigma\, d\tau\leq 2\int_0^{T'}\mathbb E_{(x,T')}[|\nabla_\sigma^\perp F|^2]\, d\sigma
	\end{align*}
	for all $x\in M$.
	\item[(4a)] {\label{thmmain3: 4a}}For any $0\leq T'\leq T$, $F\in \cyl$, we have the log-Sobolev Hessian estimate
	\begin{align*}
	&\mathbb E_{(x,T')}[F^2\log(F^2)]-\mathbb E_{(x,T')}[F^2]\log(\mathbb E_{(x,T')}[F^2])\\
	&+2\int_0^{T'}\int_0^{T'}\mathbb E_{(x,T')}[(F^2)_\tau |\nabla_\tau^\perp\nabla_\sigma^\perp \log((F^2)_\tau)|^2]\, d\sigma\, d\tau
	\leq 4\int_0^{T'}\mathbb E_{(x,T')}[|\nabla_\sigma^\perp F|^2]\, d\sigma
	\end{align*}
	for all $x\in M$.
\end{enumerate}
\end{thm12}
\begin{proof}
\eqref{thmmain3: 1} $\Rightarrow$ \eqref{thmmain3: 2}:
Assertion \eqref{thmmain3: 2} follows directly from integrating Theorem \ref{theorem: char} \eqref{thmmain: 2} and taking expectations.

\eqref{thmmain3: 2} $\Rightarrow$ \eqref{thmmain3: 6}:
Using Ito's isometry and Theorem \ref{thm: grad} yields
\begin{align*}
\E_{(x,T')}[(F_{\tau_2}-F_{\tau_1})^2]
=
&\ 2\E_{(x,T')}[\int_{\tau_1}^{\tau_2}|\nabla_\sigma^\perp F_\sigma|^2\, d\sigma]\\
\leq&\ 2\E_{(x,T')}[\int_{\tau_1}^{\tau_2}|\nabla_\sigma^\perp F|^2\, d\sigma]
= 2\E_{(x,T')}[F\, \LL_{(\tau_1,\tau_2)}F],
\end{align*}
where we used Theorem \ref{thm: grad} \eqref{thmmain2: 2} in the inequality.

\eqref{thmmain3: 6} $\Rightarrow$ \eqref{thmmain3: 1}: Dividing \eqref{thmmain3: 6} by $\tau_2-\tau_1$ and letting $\tau_2-\tau_1\to0$ we find
\begin{align*}
\E_{(x,T')}\left[\frac{d[F,F]_\tau}{d\tau}\right]\leq 2\E_{(x,T')}[|\nabla_\tau^\perp F|^2],
\end{align*}
which is Theorem \ref{thm: grad} \eqref{thmmain2: 5}.

\eqref{thmmain3: 1} $\Rightarrow$ \eqref{thmmain3: 5}:
Take $G=F^2$.
Then, by \eqref{eq: x1}
\begin{align*}
\E_{(x,T')}[(F^2)_{\tau_2}\log(F^2)_{\tau_2}-(F^2)_{\tau_1}\log(F^2)_{\tau_1}]
=&\ \E_{(x,T')}\left[\int_{\tau_1}^{\tau_2}G_{\tau}^{-1}|\nabla_\tau^\perp G_\tau|^2\, d\tau\right]\\
\leq
&\ \E_{(x,T')}\left[\int_{\tau_1}^{\tau_2}G_\tau^{-1}\E_{(x,T')}\left[|\nabla_\tau^\perp G_{T'}||\Sigma_\tau\right]^2\, d\tau\right]\\
\leq&\ 
4\E_{(x,T')}\left[\int_{\tau_1}^{\tau_2}|\nabla_\tau^\perp F|^2\, d\tau\right]\\
=&\ 4\E_{(x,T')}[F\LL_{(\tau_1,\tau_2)}F],
\end{align*}
where we used Theorem \ref{thm: grad} \eqref{thmmain2: 2} in the second step and Cauchy-Schwarz in the third.

\eqref{thmmain3: 5} $\Rightarrow$ \eqref{thmmain3: 6}: We apply \eqref{thmmain3: 5} to $F^2=1+\eps G$ and obtain by Taylor approximation
\begin{align*}
\frac12\E_{(x,T')}[\eps^2 G_{\tau_2}^2-\eps^2 G_{\tau_1}^2]\leq \eps^2\E_{(x,T')}[G\LL_{(\tau_1,\tau_2)}G]+o(\eps^2).
\end{align*}
Dividing by $\eps^2$ and letting $\eps\to0$ we obtain
\begin{align*}
\frac12\E_{(x,T')}[G_{\tau_2}^2-G_{\tau_1}^2]\leq \E_{(x,T')}[G\LL_{(\tau_1,\tau_2)}G].
\end{align*}
Noticing that $\E_{(x,T')}[G_{\tau_2}^2-G_{\tau_1}^2]=\E_{(x,T')}[(G_{\tau_2}-G_{\tau_1})^2]$ proves the claim.

This proves  the equivalence of \eqref{thmmain3: 1}-\eqref{thmmain3: 5}. Next we show the remaining implications.

\eqref{thmmain3: 2} $\Rightarrow$ \hyperref[thmmain3: 3a]{(3a)}:
Apply Proposition \ref{p:martingale} and Ito's isometry and integrate \eqref{thmmain3: 2} on $(0,T')$.

\eqref{thmmain3: 1} $\Rightarrow$ \hyperref[thmmain3: 4a]{(4a)}:
Let $G=F^2$ and consider $X_\tau=G_\tau^{-1}|\nabla^\HH G_\tau|^2-G_\tau\log(G_\tau)$.
Note that according to Corollary \ref{c: bochner3}, we have that 
\begin{align*}
dX_\tau=&\ \IP{\nabla_\tau^\perp X_\tau,dW_\tau}+2G_\tau\left(\int_0^{T'}|\nabla_\tau^\perp\nabla_\sigma^\perp \log(G_\tau)|^2\, d\sigma\right)\, d\tau\\
&\ +2G_\tau^{-1}\left(\int_0^{T'}(\Rc^{\nabla}+\frac12\partial_t(g-b))_\tau(\nabla_\tau^\perp G_\tau,\nabla_\sigma^\perp G_\tau)\, d\sigma\right)\, d\tau\\
\geq &\ \IP{\nabla_\tau^\perp X_\tau,dW_\tau}+2G_\tau\left(\int_0^{T'}|\nabla_\tau^\perp\nabla_\sigma^\perp \log(G_\tau)|^2\, d\sigma\right)\, d\tau,
\end{align*}
where we used \eqref{thmmain3: 1} in the last equation.
Integration and taking expectations yields
\begin{align*}
\E_{(x,T')}[X_T']-\E_{(x,T')}[X_0]\geq\ 2\E_{(x,T')}\left[\int_0^{T'}G_\tau\left(\int_0^{T'}|\nabla_\tau^\perp\nabla_\sigma^\perp \log(G_\tau)|^2\, d\sigma\right)\, d\tau\right],
\end{align*}
and evaluating the expectations on the left hand side
\begin{align*}
\E_{(x,T')}[X_0]=&\ -\E_{(x,T')}[F^2]\log(\E_{(x,T')}[F^2])\\
\E_{(x,T')}[X_{T'}]=&\ 4\E_{(x,T')}[|\nabla^\HH F|^2]-\E_{(x,T')}[F^2\log(F^2)].
\end{align*}
Putting everything together yields
\begin{align*}
& 4\E_{(x,T')}[|\nabla^\HH F|^2]-\E_{(x,T')}[F^2\log(F^2)]
+\E_{(x,T')}[F^2]\log(\E_{(x,T')}[F^2])\\
&\ \qquad \geq\ 2\int_0^{T'}\int_0^{T'}\E_{(x,T')}\left[(F^2)_\tau|\nabla_\tau^\perp\nabla_\sigma^\perp \log((F^2)_\tau)|^2\right]\, d\sigma\, d\tau,
\end{align*}
which is \hyperref[thmmain3: 4a]{(4a)}.

	\end{proof}

\section{Appendix: Integration by parts} \label{s:IBP}

\begin{proof}[Proof of Theorem \ref{t:IBP}]
	Since $D_V$ satisfies the product rule, it is enough to show that
\begin{align*}
\E_{(x,T')}[D_VF]=\frac12\E_{(x,T')}\left[F\int_0^{T'}\IP{\frac{d}{d\tau}h_\tau-S_\tau(\Rc^\nabla+\frac12\partial_t(g-b))_{T'-\tau}^\dagger S_\tau^{-1}h_\tau,\, dW_\tau}\right]
\end{align*}
for all $F\in\cyl$. We prove this by induction on the order $k$ of the cylinder function $F$.

$k=1$:
Let $F(\gamma)=f(x_\sigma)$ and let $s=T'-\sigma$. Since $w(x,t)=P_{t,s}f(x)$ satisfies the heat equation, the gradient satisfies
\begin{align*}
\nabla_t\grad_{g_t}w=\Delta_{g_t}w-(\Rc^\nabla+\frac12\partial_t(g_t-b_t))(\grad_{g_t}w,\cdot)^{\#_{g_t}}
\end{align*}
by Proposition \ref{p:gradientev}. By the Feynman-Kac formula (Proposition \ref{prop: feynmankac}) we have
\begin{align}\label{eq: feynmankac}
\grad_{g_{T'}}w(x,T')=\E_{(x,T')}[R_\sigma\, S_\sigma\, \grad_{g_s}f(X_\sigma)],
\end{align}
where $R_\tau=R_\tau(\gamma)\colon (T_xM,g_{T'})\to(T_xM,g_{T'})$ solves the ODE $\frac{d}{d\tau}R_\tau=R_\tau\, S_\tau\, (\Rc^\nabla+\frac12\partial_t(g-b))_{T'-\tau}S_\tau^{-1}$ with $R_0=\mathrm{id}$, and where we view $(\Rc^\nabla+\frac12\partial_t(g-b))_{T'-\tau}$ as endomorphism of $TM$ using the metric $g_{T'-\tau}$.  Note also by \eqref{eq: heatrep}
\begin{align}\label{eq: heatrep2}
f(X_\sigma)=w(x,T')+\int_0^\sigma E_i \til w(U_\tau)\, dW_\tau^i,
\end{align}
where $\til w$ is the invariant lift.

Let $(z_\tau)_{\tau\in[0,T']}\in\HH$. Then by \eqref{eq: heatrep2} and Ito's isometry
\begin{align*}
\E_{(x,T')}\left[f(X_\sigma)\int_0^\sigma\IP{R_\tau^\dagger\dot z_\tau,dW_\tau}\right]
=&\ \E_{(x,T')}\left[\int_0^\sigma E_i\til w(U_\tau)\, dW_\tau^i\, \int_0^\sigma\IP{R_\tau^\dagger\dot z_\tau,dW_\tau}\right]\\
=&\ 2\E_{(x,T')}\left[\int_0^\sigma\IP{ \nabla^E\til w(U_\tau),U_0^{-1} R_\tau^\dagger\dot z_\tau}\, d\tau\right]\\
=&\ 2\E_{(x,T')}\left[\int_0^\sigma\IP{ R_\tau U_0\nabla^E\til w(U_\tau), \dot z_\tau}_{g_T'}\, d\tau\right],
\end{align*}
where $R_\tau^\dagger$ is the transpose of $R_\tau$. 

Let $N_\tau:=R_\tau U_0\nabla^E\til w(U_\tau)=R_\tau S_\tau \grad_{g_{T'-\tau}}w(X_\tau,T'-\tau)$.
Integration by parts gives
\begin{align*}
\E_{(x,T')}\left[\int_0^\sigma\IP{ N_\tau, \dot z_\tau}_{g_{T'}}\, d\tau\right]
=\E_{(x,T')}\left[\IP{z_\sigma,N_\sigma}_{g_{T'}}-\int_0^\sigma\IP{ z_\tau, dN_\tau}_{g_{T'}}\, d\tau\right]
=\E_{(x,T')}\left[\IP{ z_\sigma, N_\sigma}_{g_{T'}}\right],
\end{align*}
where we used in the last step that $N_\tau$ is a martingale, cf. equation \eqref{eq: heatrep3}. With this we obtain
\begin{align*}
\E_{(x,T')}\left[f(X_\sigma)\int_0^{T'}\IP{ R_\tau^\dagger \dot z_\tau, dW_\tau}\right]
=2\E_{(x,T')}\left[\IP{ R_\sigma^\dagger z_\sigma, S_\sigma \grad_{g_\sigma} f(X_\sigma)}_{g_{T'}}\right],
\end{align*}
where we used that $\E_{(x,T')}\left[f(X_\sigma)\int_\sigma^{T'}\IP{R_\tau^\dagger\dot z_\tau,dW_\tau}\right]=0$.
Finally, let $h_\tau=R_\tau^\dagger z_\tau$. Then
\begin{align*}
R_\tau^\dagger\dot z_\tau=\dot h_\tau-S_\tau(\Rc^\nabla+\frac12\partial_t(g-b))^\dagger_{T'-\tau}S_\tau^{-1}h_\tau,
\end{align*}
wich is the claim.

Now we prove the inductive step, assuming the result for $k-1$-point cylinder functions.  Let $F(\gamma)=f(x_{\sigma_1},\ldots, x_{\sigma_k})$ and let $s_i=T'-\sigma_i$. Then
\begin{align}\label{eq: intbyparts1}
\E_{(x,T')}[D_VF]= \sum_{j=1}^k\E_{(x,T')}\IP{h_{\sigma_j},S_{\sigma_j}\grad_{g_{s_j}}^{(j)}f(X_{\sigma_1},\ldots,X_{\sigma_k})}_{g_{T'}}.
\end{align}
We define a function 
\begin{align*}
\alpha(x_1,\ldots,x_{k-1}):=\E_{(x_{k-1},s_{k-1})}f(x_1,\ldots,x_{k-1},X_{\sigma_k-\sigma_{k-1}}'),
\end{align*}
where $X'$ is based at $x_{k-1}$. Then, for $j=1,\ldots,k-2$, we have
\begin{align}\label{eq: intbyparts2}
\grad_{g_{s_j}}^{(j)}\alpha(x_1,\ldots,x_{k-1})=\E_{(x_{k-1},s_{k-1})}\grad_{g_{s_j}}^{(j)}f(x_1,\ldots,x_{k-1},X_{\sigma_k-\sigma_{k-1}}').
\end{align}
For $j=k-1$ we have by the product rule and \eqref{eq: feynmankac}
\begin{gather}\label{eq: intbyparts3}
\begin{split}
\grad_{g_{s_{k-1}}}^{(k-1)}\alpha(x_1,\ldots,x_{k-1})=&\ \E_{(x_{k-1},s_{k-1})}\grad_{g_{s_{k-1}}}^{(k-1)}f(x_1,\ldots,x_{k-1},X_{\sigma_k-\sigma_{k-1}}')\\
&\ +\E_{(x_{k-1},s_{k-1})}[R'_{\sigma_k-\sigma_{k-1}}S'_{\sigma_k-\sigma_{k-1}}\grad_{g_{s_k}}f(x_1,\ldots,x_{k-1},X_{\sigma_k-\sigma_{k-1}}')],
\end{split}
\end{gather}
where $R_\tau'=R_\tau'(\gamma)\colon (T_{x_{k-1}}M,g_{s_{k-1}})\to(T_{x_{k-1}}M,g_{s_{k-1}})$ solves the ODE
$\frac{d}{d\tau}R_\tau'=R_\tau'S_\tau'(\Rc^\nabla+\frac12\partial_t(g-b))_{s_{k-1}-\tau}S_\tau'^{-1}$ with $R_0=\mathrm{id}$. Now, let $\G\colon P_{(x,T')}\MM\to\R$ be the $(k-1)$-point cylinder function induced by $\alpha$
\begin{align*}
G(\gamma)=\alpha(x_{\sigma_1},\cdots,x_{\sigma_{k-1}}).
\end{align*}
Then, by \eqref{eq: intbyparts1}, \eqref{eq: intbyparts2}, and \eqref{eq: intbyparts3}, and by the law of total expectation
\begin{equation}
\begin{aligned}\label{eq: intbyparts4}
&\E_{(x,T')}[D_VF]=\E_{(x,T')}[D_VG]+\E_{(x,T')}\left[\IP{h_{\sigma_k},S_{\sigma_k}\grad_{g_{s_k}}^{(k)}f(X_{\sigma_1},\ldots,X_{\sigma_k})}_{g_{T'}}\right]\\
&-\E_{(x,T')}\left[\E_{(X_{\sigma_{k-1}},s_{k-1})}\left[\IP{h_{\sigma_{k-1}},S_{\sigma_{k-1}}R'_{\sigma_k-\sigma_{k-1}}
	S'_{\sigma_k-\sigma_{k-1}}\grad_{g_{s_k}}f(X_{\sigma_1},\ldots,X_{\sigma_{k-1}},X_{\sigma_k-\sigma_{k-1}}')}_{g_{T'}}\right]\right]
\end{aligned}
\end{equation}

By induction, the claim holds for $k-1$ and thus
\begin{align}
\E_{(x,T')}[D_VG]=\frac12\E_{(x,T')}\left[G\int_0^{s_{k-1}}\IP{\frac{d}{d\tau}h_\tau-S_\tau(\Rc^\nabla+\frac12\partial_t(g-b))_{T'-\tau}^\dagger S_\tau^{-1}h_\tau,\, dW_\tau}\right].\label{eq: add3}
\end{align}
The second term on the right hand side of \eqref{eq: intbyparts4} we decompose into
\begin{equation}\begin{aligned}\label{eq: markov}
\E_{(x,T')}& \left[\IP{h_{\sigma_k},S_{\sigma_k}\grad_{g_{s_k}}^{(k)}f(X_{\sigma_1},\ldots,X_{\sigma_k})}_{g_{T'}}\right]\\
=&\
\E_{(x,T')}\left[\IP{h_{\sigma_k}-h_{\sigma_{k-1}},S_{\sigma_k}\grad_{g_{s_k}}^{(k)}f(X_{\sigma_1},\ldots,X_{\sigma_k})}_{g_{T'}}\right] +\E_{(x,T')}\left[\IP{h_{\sigma_{k-1}},S_{\sigma_k}\grad_{g_{s_k}}^{(k)}f(X_{\sigma_1},\ldots,X_{\sigma_k})}_{g_{T'}}\right].
\end{aligned}\end{equation}
Then, for the first term in \eqref{eq: markov} the Markov property at $(X_{\sigma_{k-1}},s_{k-1})$ together with the induction hypothesis for one-point cylinder functions imply
\begin{gather} \label{eq: add2}
\begin{split}
&\E_{(x,T')}\left[\IP{h_{\sigma_k}-h_{\sigma_{k-1}},S_{\sigma_k}\grad_{g_{s_k}}^{(k)}f(X_{\sigma_1},\ldots,X_{\sigma_k})}_{g_{T'}}\right]\\
=&\ \E_{(x,T')}\E_{(X_{\sigma_{k-1}},s_{k-1})}\left[\IP{S_{\sigma_{k-1}}^{-1}(h_{\sigma_k}-h_{\sigma_{k-1}}),S_{\sigma_k-\sigma_{k-1}}'\grad_{g_{s_k}}^{(k)}f(X_{\sigma_1},\ldots,X_{\sigma_{k-1}},X_{\sigma_k-\sigma_{k-1}}')}_{g_{s_{k-1}}}\right]\\
=&\ \frac12\E_{(x,T')}\left[F\int_{s_{k-1}}^{s_k}\IP{\frac{d}{d\tau}h_\tau-S_\tau(\Rc^\nabla+\frac12\partial_t(g-b))^\dagger S_\tau^{-1}(h_\tau-h_{\sigma_{k-1}}),dW_\tau}\right].
\end{split}
\end{gather}
Similarly for the other term in \eqref{eq: markov}
\begin{align*}
&\E_{(x,T')}\left[\IP{h_{\sigma_{k-1}},S_{\sigma_k}\grad_{g_{s_k}}^{(k)}f(X_{\sigma_1},\ldots,X_{\sigma_k})}_{g_{T'}}\right]\\
=&\ \E_{(x,T')}\E_{(X_{\sigma_{k-1}},s_{k-1})}\left[\IP{S_{\sigma_{k-1}}^{-1}h_{\sigma_{k-1}},S_{\sigma_k-\sigma_{k-1}}'\grad_{g_{s_k}}^{(k)}f(X_{\sigma_1},\ldots,X_{\sigma_{k-1}},X_{\sigma_k-\sigma_{k-1}}')}_{g_{s_{k-1}}}\right],
\end{align*}
which we combine with the third term on the right hand side in \eqref{eq: intbyparts4} and obtain
\begin{gather} \label{eq: add1}
\begin{split}
&\E_{(x,T')}\E_{(X_{\sigma_{k-1}},s_{k-1})}\left[\IP{S_{\sigma_{k-1}}^{-1}h_{\sigma_{k-1}},S_{\sigma_k-\sigma_{k-1}}'\grad_{g_{s_k}}^{(k)}f(X_{\sigma_1},\ldots,X_{\sigma_{k-1}},X_{\sigma_k-\sigma_{k-1}}')}_{g_{s_{k-1}}}\right]\\
&-\E_{(x,T')}\E_{(X_{\sigma_{k-1}},s_{k-1})}\left[\IP{h_{\sigma_{k-1}},S_{\sigma_{k-1}}R'_{\sigma_k-\sigma_{k-1}}
	S'_{\sigma_k-\sigma_{k-1}}\grad_{g_{s_k}}f(X_{\sigma_1},\ldots,X_{\sigma_{k-1}},X_{\sigma_k-\sigma_{k-1}}')}_{g_{T'}}\right]\\
=&\ \E_{(x,T')}\E_{(X_{\sigma_{k-1}},s_{k-1})}\left[\IP{(\mathrm{id}-R_{\sigma_{k}-\sigma_{k-1}}')^\dagger S_{\sigma_{k-1}}^{-1}h_{\sigma_{k-1}},S_{\sigma_k-\sigma_{k-1}}'\grad_{g_{s_k}}f(X_{\sigma_1},\ldots,X_{\sigma_{k-1}},X_{\sigma_k-\sigma_{k-1}}')}_{g_{s_{k-1}}}\right]\\
=&\ -\frac12\E_{(x,T')}\left[F\int_{s_{k-1}}^{s_k}\IP{S_{\tau}(\Rc^\nabla+\frac12\partial_t(g-b))^\dagger S_{\tau}^{-1}h_{\sigma_{k-1}},dW_\tau}\right].
\end{split}
\end{gather}
Here, we used the induction hypothesis for
\begin{align*}
w_\tau:=(\mathrm{id}-R_{\tau-\sigma_{k-1}}'^\dagger)S_{\sigma_{k-1}}^{-1}h_{\sigma_{k-1}}.
\end{align*}
Then a simple calculation shows that
\begin{align*}
\frac{d}{d\tau}w_\tau-S_{\tau-\sigma_{k-1}}'(\Rc^\nabla+\frac12\partial_t(g-b))^\dagger S_{\tau-\sigma_{k-1}}'^{-1}w_\tau=-S_{\tau-\sigma_{k-1}}'(\Rc^\nabla+\frac12\partial_t(g-b))^\dagger S_{\tau-\sigma_{k-1}}'^{-1}S_{\sigma_{k-1}}^{-1}h_{\sigma_{k-1}},
\end{align*}
which gives \eqref{eq: add1}.

Adding \eqref{eq: add3}, \eqref{eq: add2}, \eqref{eq: add1}, we obtain
\begin{align*}
\E_{(x,T')}[D_VF]=&\ \frac12\E_{(x,T')}\left[G\int_0^{s_{k-1}}\IP{\frac{d}{d\tau}h_\tau-S_\tau(\Rc^\nabla+\frac12\partial_t(g-b))_{T'-\tau}^\dagger S_\tau^{-1}h_\tau,\, dW_\tau}\right]\\
&+\frac12\E_{(x,T')}\left[F\int_{s_{k-1}}^{s_k}\IP{\frac{d}{d\tau}h_\tau-S_\tau(\Rc^\nabla+\frac12\partial_t(g-b))^\dagger S_\tau^{-1}(h_\tau-h_{\sigma_{k-1}}),dW_\tau}\right]\\
&-\frac12\E_{(x,T')}\left[F\int_{s_{k-1}}^{s_k}\IP{S_{\tau}(\Rc^\nabla+\frac12\partial_t(g-b))^\dagger S_{\tau}^{-1}h_{\sigma_{k-1}},dW_\tau}\right]\\
=&\ \frac12\E_{(x,T')}\left[F\int_0^{T'}\IP{\frac{d}{d\tau}h_\tau-S_\tau(\Rc^\nabla+\frac12\partial_t(g-b))_{T'-\tau}^\dagger S_\tau^{-1}h_\tau,\, dW_\tau}\right],
\end{align*}
which proves the theorem.
\end{proof}
	
%\bibliography{KopferStreetsBFFIGRF}
%\bibliographystyle{plain}

\end{document}